\documentclass[12pt, twoside]{article}
\usepackage{amsfonts}
\usepackage{xcolor}
\usepackage{mathrsfs}
\usepackage[all,cmtip]{xy}
\usepackage{amsmath}
\usepackage{amssymb,amsthm,upref,amscd}
\usepackage{setspace}
\usepackage{enumerate}
\usepackage[titletoc]{appendix}
\usepackage{times}
\usepackage{cite}
\usepackage{tikz}
\usepackage[colorinlistoftodos]{todonotes}
\usetikzlibrary{arrows}
\numberwithin{equation}{section}

\def\titlerunning#1{\gdef\titrun{#1}}
\makeatletter
\def\author#1{\gdef\autrun{\def\and{\unskip, }#1}\gdef\@author{#1}}

\makeatother

\def\keywords#1{\par\medskip
\noindent\textbf{Keywords.} #1}

\allowdisplaybreaks

\theoremstyle{plain}
\newtheorem{Thm}{Theorem}[section]
\newtheorem{Lem}[Thm]{Lemma}

\newtheorem{Cor}[Thm]{Corollary}
\newtheorem{Prop}[Thm]{Proposition}
\newtheorem*{Thm*}{Theorem}
\newtheorem*{claim*}{Claim}
\theoremstyle{definition}

\newtheorem*{Def*}{Definition}
\newtheorem*{Cor*}{Corollary}
\newtheorem{Rem}[Thm]{Remark}

\newtheorem*{Ques*}{Question}

\newcommand{\equ}{equation}
\newcommand{\C}{\mathbb{C}}
\newcommand{\N}{\mathbb{N}}
\newcommand{\R}{\mathbb{R}}
\newcommand{\Z}{\mathbb{Z}}

 \DeclareMathOperator{\dist}{dist}

\DeclareMathOperator{\supp}{supp}

\DeclareMathOperator{\vol}{vol}

\let\nhatoksa=\theenumi
\let\nhatoksb=\labelenumi
\let\nhatoksc=\theenumii
\let\nhatoksd=\labelenumii
\newlength{\nhalengtha}
\newlength{\nhalengthb}
\newlength{\nhalengthc}
\newcommand{\resetenum}{
\let\theenumi=\nhatoksa
\let\labelenumi=\nhatoksb
\let\theenumii=\nhatoksc
\let\labelenumii=\nhatoksd
\setlength{\leftmargini}{\nhalengtha}
\setlength{\leftmarginii}{\nhalengthb}
\setlength{\labelwidth}{\nhalengthc}
}

\newcommand\cc{\mathcal{C}}

\newcommand\ce{\mathcal{E}}

\newcommand\ch{\mathcal{H}}

\newcommand\cl{\mathcal{L}}
\newcommand\cm{\mathcal{M}}

\def\mbs{\mathbb{S}}

\def\msh{\mathscr{H}}

\def\msn{\mathscr{N}}

\def\id{\text{Id}}
\def\ig{\textit{g}}

\def\ov{\overline}

\def\pa {\partial}

\def\op{\oplus}
\def\ot{\otimes}

\def\ka{\kappa}
\def\al{\alpha}
\def\bt{\beta}

\def\de{\delta}
\def\Ga{\Gamma}
\def\ga{\gamma}

\def\lm{\lambda}

\def\om{\omega}

\def\sa{\sigma}

\def\vr{\varepsilon}
\def\va{\varphi}

\def\span{\hbox{span}}

\def\vol{\mathrm{vol}}
\def\real{\hbox{Re}}

\newcommand{\inp}[2]{\left\langle#1,#2\right\rangle}

\def\clf{\C\ell}

\begin{document}

\titlerunning{The spinorial Yamabe problem in product manifolds}

\title{A variational analysis of the spinorial Yamabe equation on product manifolds}

\author{Yannick Sire,\, Tian Xu\footnote{\noindent
 Supported by the National Science Foundation of China (NSFC 11601370)
and the Alexander von
\newline \hspace*{1.5em}
Humboldt Foundation of Germany}}

\date{}

\maketitle

\begin{abstract}
This work is devoted to the analysis of the Yamabe problem on Spin manifolds and some applications to CMC immersions. Despite the efforts of many authors, very little is known on the existence of Yamabe metrics on general Spin manifolds. Motivated to bubbling phenomena for the Riemannian problem and recent multiplicity results in this setting, we investigate special spinorial Yamabe metrics on product manifolds developing a bubbling analysis which has independent interest in the present setting. 



\end{abstract}

\tableofcontents

\section{Introduction and setting of the problem}

The well known Yamabe problem seeks for the existence of a constant scalar curvature metric in a given conformal class of Riemannian metrics on a compact manifold. Such a metric can be characterized variationally as a critical point of the Hilbert-Einstein functional on conformal classes. A positive answer to this problem, obtained in a series of steps by H. Yamabe \cite{Yamabe}, T. Aubin \cite{Aubin}, N. Trudinger \cite{Trudinger} and R. Schoen \cite{Schoen1984}, provides at least one minimizer of the Hilbert-Einstein functional in each conformal class. Such a metric is called a Yamabe metric. To determine all the Yamabe metrics in a given conformal
class is generally a very difficult problem particularly when the scalar curvature has positive sign. It is interesting to observe that, generically, minima of the Hilbert-Einstein functional in conformal classes are unique, see \cite{Anderson}. However, in many cases a rich variety of constant scalar curvature metrics arise as critical points that are not necessarily minimizers, and it is a very interesting task to classify those metrics. Multiplicity of solutions of the Yamabe problem has been studied in the literature, especially in product manifolds, several results have been obtained in the special case of products with round spheres, see for instance \cite{HV, Kobayashi, Petean, Schoen1989} and references therein.

In the setting of Spin Geometry, a problem analogous to the Yamabe problem has received increasing attention in recent years.
Several works of Ammann \cite{Ammann, Ammann 2003, Ammann2009} and Ammann, Humbert {\sl et al}  \cite{AGHM, AHA, AHM} provide a framework in which variational methods may be employed.

Let $(M,\ig,\sa)$ be an $m$-dimensional closed spin manifold with a metric $\ig$, a spin structure $\sa:P_{Spin}(M)\to P_{SO}(M)$. With the notation $\rho: Spin(m)\to End(\mbs_m)$ be the spin representation,  we denote  $\mbs(M)=P_{Spin}(M)\times_{\rho}\mbs_m$ the spinor bundle over $M$ and $D_\ig^M: C^\infty(M,\mbs(M))\to C^\infty(M,\mbs(M))$ the Dirac operator (see \cite{Friedrich, Lawson} for more geometric backgrounds). Analogous to the Yamabe invariant, a spin conformal invariant is defined as
\begin{\equ}\label{BHL}
\lm_{min}^+(M,[\ig],\sa):=\inf_{\tilde\ig\in[\ig]}\lm_1^+(\tilde\ig)
\text{Vol}(M,\tilde\ig)^{\frac1m}
\end{\equ}
where $\lm_1^+(\tilde\ig)$ denotes the smallest positive eigenvalue of the Dirac operator $D_{\tilde\ig}^M$ with respect to the conformal metric $\tilde\ig\in[\ig]:=\big\{ f^2\ig:\, f\in C^{\infty}(M),\, f>0 \big\}$. Ammann points out in \cite{Ammann, Ammann2009} that studying critical metrics for this invariant involves similar analytic problems to those appearing in the Yamabe problem. It follows that finding a critical metric of \eqref{BHL} is  equivalent to prove the
existence of a spinor field $\psi\in C^{\infty}(M,\mbs(M))$ minimizing the functional defined by
\begin{\equ}\label{J-ig-functional}
J_\ig(\phi)=\frac{\Big(
\int_M|D_{\ig}^M\phi|^{\frac{2m}{m+1}}d\vol_{\ig}
\Big)^{\frac{m+1}{m}}}{\big|\int_M(D_{\ig}^M\phi,\phi)d\vol_{\ig}\big|}
\end{\equ}
with the Euler-Lagrange equation
\begin{\equ}\label{Dirac-c}
D_\ig^M\psi=\lm_{min}^+(M,[\ig],\sa) |\psi|_\ig^{2^*-2}\psi,
\end{\equ}
where $m^*:=\frac{2m}{m-1}$.

As was pointed out in \cite{Ammann}, standard variational method does not imply the existence of minimizers for $J_\ig$ directly. This is due to the  criticality of the nonlinearity in \eqref{Dirac-c}. Indeed, the exponent $m^*=\frac{2m}{m-1}$ is critical for the corresponding Sobolev embedding. Similar to the argument in solving the Yamabe problem, one might be able to find a criterion which recovers the compactness. It is crucial to note that a spinorial analogue of Aubin's inequality holds (see \cite{AGHM})
\begin{\equ}\label{spinorial Aubin inequ}
\lm_{min}^+(M,[\ig],\sa)\leq \lm_{min}^+(S^m,[\ig_{S^m}],\sa_{S^m})=\frac m2 \om_m^{\frac1m}
\end{\equ}
where $(S^m,\ig_{S^m},\sa_{S^m})$ is the $m$-dimensional sphere equipped with its canonical metric $\ig_{S^m}$ and its standard spin structure $\sa_{S^m}$, and $\om_m$ is the standard volume of $(S^m,\ig_{S^m})$. The criterion obtained in \cite{Ammann} shows that if inequality \eqref{spinorial Aubin inequ} is strict then the spinorial Yamabe problem \eqref{Dirac-c} has a nontrivial solution minimizing the functional $J_\ig$. However, the strict inequality in \eqref{spinorial Aubin inequ} is only verified for some special cases and general results are still missing (cf. \cite{ADHH, AHM, Ginoux}).

Tightly related to geometric data, the nonlinear problem \eqref{Dirac-c} provides a strong tool for showing the existence of constant mean curvature hypersurfaces in Euclidean spaces. This has been known as the \textit{Spinorial Weierstra\ss\ representation}, see for instance \cite{KuSt, Friedrich1998, Ammann}. This is one of the most attractive features of the spinorial Yamabe problem that unseals new researches in both PDE theory and Riemannian geometry.
Not being confined by the strict inequality in \eqref{spinorial Aubin inequ}, the purpose of this paper is to establish some existence of multiple solutions of \eqref{Dirac-c} on products of compact spin manifolds which are not necessarily minimizers of the functional defined in \eqref{J-ig-functional}.

Let us describe our results more precisely. Given closed spin manifolds $(M_1,\ig^{(1)},\sa_1)$ and $(M_2,\ig^{(2)},\sa_2)$, with fixed spin structures, we consider a family of metrics $\ig_\ell$ on the product $N=M_1\times M_2$ defined by $\ig_\ell=\ell^2\ig^{(1)}\op \ig^{(2)}$, $\ell>0$. Noting $m_1$ and $m_2$  the dimensions of $M_1$ and $M_2$ respectively, we will be interested in solutions of the (normalized) spinorial Yamabe equation on the product manifold $(N,\ig_\ell)$:
\begin{\equ}\label{Dirac-product}
D_{\ig_\ell}^N\phi=|\phi|_{\ig_\ell}^{\frac2{n-1}}\phi
\end{\equ}
where $n:=\dim N=m_1+m_2$. For simplicity we first describe here the case $n$ is odd, but the statement is similar for $n$ even, see Theorem \ref{result for n geq 3}. Without loss of generality, we may assume $m_1$ is even (otherwise, it is equivalent to consider the product manifold equipped with the metrics $\ig^{(1)}\op\ell^2\ig^{(2)}$). In this setting, the spinor bundle over $N$ can be identified with $\mbs(N)=\mbs(M_1)\otimes\mbs(M_2)$ and the Dirac operator is given by
\[
D_{\ig_\ell}^N(\psi\otimes\va)=D^{M_1}_{\ell^2\ig^{(1)}}\psi\otimes\va
+\om_\C^{M_1}\cdot_{\mbox{\tiny $\ell^2\ig^{(1)}$}}\psi\otimes D_{\ig^{(2)}}^{M_2}\va
\]
for $\psi\in C^{\infty}(M_1,\mbs(M_1))$ and $\va\in C^{\infty}(M_2,\mbs(M_2))$, where $D^{M_1}_{\ell^2\ig^{(1)}}$ and $D_{\ig^{(2)}}^{M_2}$ denote the Dirac operators on $M_1$ and $M_2$ respectively, $\om_\C^{M_1}$ is the chirality operator in the Clifford bundle over $M_1$ and "$\cdot_{\mbox{\tiny $\ell^2\ig^{(1)}$}}$" is the representation of Clifford multiplication on the spinor bundle $\mbs(M_1)$ with respect to the metric $\ell^2\ig^{(1)}$ (see Section \ref{preliminaries} for detailed definitions of these notations). One can use a specials ansatz for the solutions $\phi=\psi\otimes\va_{\lm}$, $\lm>0$, with $\va_\lm$ being an eigenspinor of the second factor, i.e. $D_{\ig^{(2)}}^{M_2}\va_{\lm}=\lm \va_{\lm}$. A further assumption on such an ansatz is that the component $\va_\lm$ is normalizable in the sense one can normalize $\va_\lm$ so that $|\va_\lm|_{\ig^{(2)}}\equiv1$ almost everywhere with respect to the canonical measure on $(M_2,\ig^{(2)})$. Then $\phi=\psi\otimes\va_{\lm}$ solves the Yamabe equation \eqref{Dirac-product} if and only if $\psi$ solves
\begin{\equ}\label{reduced-Dirac}
D^{M_1}_{\ell^2\ig^{(1)}}\psi+\lm\om_\C^{M_1}\cdot_{\mbox{\tiny $\ell^2\ig^{(1)}$}}\psi=|\psi|_{\ell^2\ig^{(1)}}^{\frac2{m-1}}\psi
\quad \text{on } M_1.
\end{\equ}

One of the motivations of the present work comes from the geometric bifurcation theory in order to prove multiplicity and qualitative behaviour of Riemannian Yamabe metrics. This problem goes back to the seminal work of Schoen \cite{Schoen1989} and has been deeply investigated in several works (see e.g. \cite{BP1,BP2,BPS} and references therein). Together with Bettiol and Piccione \cite{BPSire}, the first author investigated the multiplicity of constant $Q-$curvature metrics in a similar product manifold as in our setting but in the framework of Berger spheres. In there, the authors considered Gromov-Hausdorff limits of Einstein Riemannian minimal submersions under a parameter $\ell$. The method used in geometric bifurcation theory is the use of a general bifurcation result that detects bifurcations under a jump of the Morse index of the Jacobi operator. From this point of view, our problem is substantially more difficult than the Riemannian case since the functional is strongly indefinite, i.e. the linearization has an infinite number of eigenvalues and such a criterion for bifurcating solutions cannot be used.  Another issue in the spinorial setting is that there is no significant difference in showing one solution and multiple solutions, unless one can find a way to distinguish these solutions (not simply by energies). The method developed here allows actually to distinguish the solutions.

We will find solutions to the spinorial Yamabe equation \eqref{Dirac-product} by solving \eqref{reduced-Dirac}. The point here is that the problem becomes subcritical and the existence of solutions is easy to prove by variational techniques. Our strategy is the following: 

\begin{itemize}
\item We first construct for an infinite number of values of the parameter $\ell$  a solution to \eqref{reduced-Dirac}.
\item Then, we investigate the bubbling phenomenon as $\ell \to \infty$.
\item Coming back to the original problem \eqref{Dirac-product}, a quantization formula allows to distinguish each of these solutions.
\end{itemize}

We now state our main results. The first one is the description in dimension greater or equal to $3$ of the infinite family of solutions of the spinorial Yamabe problem under consideration (see Theorem \ref{result for n geq 3} for a more precise statement) .

\begin{Thm}\label{mainIntro1}
There exists $\ell_0>0$ (possibly depending on $\lm$) such that for any $\ell>\ell_0$ there is a non-trivial solution $\psi_\ell$ of \eqref{reduced-Dirac} which is highly concentrated in the sense that, as $\ell\to\infty$, there exists a converging sequence $\xi_\ell\to\xi_0\in M_1$ such that $|\psi_\ell(\xi_\ell)|\to+\infty$ and $|\psi_\ell|\to0$ uniformly on compact subsets of $M_1\setminus\{\xi_0\}$.

Furthermore, the spinor field $\phi_{\ell,\lm}:=\psi_\ell\otimes\va_\lm$ defines a generalized conformal metric $\ig_{\ell,\lm}=|\phi_{\ell,\lm}|_{\ig_\ell}^{\frac4{n-1}}\ig_{\ell}$ on $N=M_1\times M_2$ (in the sense of Ammann \cite[Section 3]{Ammann2009}) such that
    \[
    \lim_{\ell\to\infty}\frac{\text{Vol}(N,\ig_{\ell,\lm})}{\lm^{m_2}}=C_{M_1,M_2}
    \]
    where $C_{M_1,M_2}>0$ is a constant depending only on $(M_1,\ig^{(1)})$ and $(M_2,\ig^{(2)})$.
\end{Thm}

An interesting corollary of the above statements is for the case $N=M_1\times S^{1}$. Since all eigenspinors on $S^1$ are normalizable, substitute different eigenvalues of $\lm$ in \eqref{reduced-Dirac}, one obtains multiple solutions for large $\ell$ as the volume functional $\text{Vol}(N,\ig_{\ell,\lm})$ can be distinguished by varying the values of $\lm$.

The case $m=2$ is of particular interests because we are concerned with the $2$-dimensional torus, i.e. $N=S^1\times S^1$ equipped with the family of metrics $\ig_\ell:=\ell^2 d^2t\op d^2\tau$ with $(t,\tau)\in[0,2\pi]\times[0,2\pi]$ being the standard parameterizations.

Our second theorem is an application of the previous analysis. We refer the reader to Theorem \ref{thm tours} for a precise statement.
\begin{Thm}\label{mainIntro2}
There exists  a non-trivial solution of constant length $\lm$ of \eqref{reduced-Dirac} for all $\ell>0$.

Furthermore, at all value $\ell^*\in\big\{\frac1{2\lm},\frac2{2\lm},\dots,\frac n{2\lm},\dots\big\}$, there is a bifurcating branch of solutions issuing from the constant length solution branch, and these branches consist of periodic solutions that do not have the same fundamental period;

Finally if $N=S^1\times S^1$ is equipped with the so-called non-trivial spin structure $\sa_N^*$ then the strict inequality in \eqref{spinorial Aubin inequ} is valid, i.e.
    \[
   \lm_{\min}^+(N,\ig_\ell,\sa_N^*)<2\sqrt\pi
    \]
    for all $\ell>0$.

\end{Thm}

The assumption that there exists normalizable eigenspinor on $M_2$ is rather harmless. This is satisfied by a large class of manifolds including the circle, the spheres and any spin $m$-manifold which can be immersed into $\R^{m+1}$ with constant mean curvature (see \cite[Chapter 5]{Ammann} and \cite{ADHH, AHM}).

\section{The setting}\label{preliminaries}

\subsection{Some algebraic preliminaries}\label{algebraic settings}

Our aim is to derive the Dirac operator on Riemannian products of spin manifolds.
In particular, we have to compare the spinor bundle of the ambient space with
the spinor bundles of the factor manifolds. The starting point is the splitting of
the tangent bundle of the large manifold into direct sum of two vector bundles associated
with the two factors.  Instructional materials can be found in \cite[Chapter I. 5 and II. 7]{Lawson},
but here we want to make it more explicit.

Let us denote by $\{e_1,\dots, e_m\}$ the canonical basis of an oriented
Euclidean space $V$ and by $\clf(V)$ the complex Clifford algebra of $V$
with its multiplication being denoted by "$\cdot$".
In case the dimension $m$ of $V$ is even, i.e. $m=2k$, the Clifford algebra
is isomorphic to the algebra $\cm(2^k;\C)$ of all complex matrices of rank $2^k$. Hence
$\clf(V)$ has precisely one irreducible module, the spinor module $\mbs_{2k}$ with $\dim\mbs_{2k}=2^k$.
For ease of notations, we simply write the Clifford representation as
\[
\clf(V)\ot\mbs_{2k}\to  \mbs_{2k}, \quad
\xi\ot \psi\mapsto \xi\cdot\psi.
\]
When restricted this representation to the even subalgebra $\clf^0(V)$,
the module $\mbs_{2k}$ splits into two irreducible unitary representations
$\mbs_{2k}=\mbs_{2k}^+\op\mbs_{2k}^-$, given by the eigensubspaces
of the endomorphism $\om_\C:=i^k e_1\cdots e_m$ to the eigenvalues $\pm1$.
In the sequel, we can call $\om_\C$ the "chirality operator" or the
"complex volume element".

In case $m$ is odd, that is $m=2k+1$, the Clifford algebra $\clf(V)$ is isomorphic
to $\cm(2^k;\C)\op\cm(2^k;\C)$. And thus, we obtain two $2^k$-dimensional
irreducible spinor modules $\mbs^0_{2k+1}$ and $\mbs^1_{2k+1}$ if we project the
Clifford multiplication onto the first and second component respectively.
Similar to the splitting in even dimensions, the two modules $\mbs^0_{2k+1}$
and $\mbs^1_{2k+1}$ can be distinguished by the action of the chirality operator
$\om_\C:=i^{k+1} e_1\cdots e_m$ in the sense that on $\mbs^j_{2k+1}$ it acts as
$(-1)^j$, $j=0,1$. It will cause no confusion if we simply identify $\mbs^0_{2k+1}$
and $\mbs^1_{2k+1}$ as the same vector space, that is
$\mbs_{2k+1}=\mbs^0_{2k+1}=\mbs^1_{2k+1}$, and equip them with
Clifford multiplications of opposite sign.

Now let $V$ and $W$ be two oriented Euclidean spaces with $\dim V=m_1$
and $\dim W=m_2$. We denote $\clf(V)$ and $\clf(W)$ the associated Clifford
algebras of $V$ and $W$ respectively. By abuse of notation, we use the same
symbol "$\cdot$" for the Clifford multiplication in $\clf(V)$, $\clf(W)$
and in their representations.
As is well known, the Clifford algebra of the sum of two vector spaces is the
$\Z_2$-graded tensor product of the Clifford algebras of the two summands,
that is $\clf(V\op W)=\clf(V)\widehat\ot\clf(W)$ (see \cite{Lawson}).
Therefore, we can construct the spinor module of $V\op W$ from those of
$V$ and $W$ as
\begin{\equ}\label{spinor modules}
\mbs_{m_1+m_2}=\left\{
\aligned
&(\mbs_{m_1}\op\mbs_{m_1})\ot \mbs_{m_2} &\quad & \text{both } m_1 \text{ and } m_2 \text{ are odd}, \\
&\qquad \mbs_{m_1}\ot\mbs_{m_2} &\quad & m_1 \text{ is even}.
\endaligned \right.
\end{\equ}
Here, we have excluded the case $m_1$ is odd and $m_2$ is even. This is simply because
the place of $V$ and $W$ can be interchanged, which suggests that this case is symmetric to the case $m_1$ is even and $m_2$ is odd.
As for the representation of Clifford multiplications on $\mbs_{m_1+m_2}$, let $\xi\in V$, $\zeta\in W$, $\va\in\mbs_{m_2}$
and $\psi=\psi_1\op\psi_2\in \mbs_{m_1}\op\mbs_{m_1}$ for both $m_1$ and $m_2$ are odd and $\psi\in\mbs_{m_1}$ otherwise, we set
\begin{\equ}\label{clifford multiplication}
(\xi\op\zeta)\cdot (\psi\ot\va)=(\xi\cdot\psi)\ot\va+(\om_\C^V\cdot\psi)\ot (\zeta\cdot\va),
\end{\equ}
where for both $m_1$ and $m_2$ odd we set $\xi\cdot\psi=(\xi\cdot\psi_1)\op(-\xi\cdot\psi_2)$
and $\om_\C^V\cdot\psi=i(\psi_2\op-\psi_1)$. With this notation, one easily checks
\[
(\xi\op\zeta)\cdot(\xi\op\zeta)\cdot (\psi\ot\va)=-|\xi\op\zeta|^2(\psi\ot\va).
\]
Thus $\mbs_{m_1+m_2}$ is a nontrivial $\clf(V\op W)$-module of dimension
$2^{[\frac{m_1+m_2}{2}]}$. Moreover, in case $m_1+m_2$ is even, the splitting of
$\mbs_{m_1+m_2}$ into half-spinor modules is given by
\[
\mbs_{m_1+m_2}^+=\big\{ (\psi\op\psi)\ot\va:\, \psi\in\mbs_{m_1},\ \va\in\mbs_{m_2}\big\},
\]
\[
\mbs_{m_1+m_2}^-=\big\{ (\psi\op-\psi)\ot\va:\, \psi\in\mbs_{m_1},\ \va\in\mbs_{m_2}\big\}
\]
for both $m_1$ and $m_2$ odd and
\[
\mbs_{m_1+m_2}^+=(\mbs_{m_1}^+\ot\mbs_{m_2}^+)\op(\mbs_{m_1}^-\ot\mbs_{m_2}^-),
\]
\[
\mbs_{m_1+m_2}^-=(\mbs_{m_1}^+\ot\mbs_{m_2}^-)\op(\mbs_{m_1}^-\ot\mbs_{m_2}^+)
\]
for both $m_1$ and $m_2$ even.

\begin{Rem}
The construction of the Clifford multiplication over $\mbs_{m_1+m_2}$ is a subtle
issue. Comparing with the explicit formula \eqref{clifford multiplication}, there are different
ways to define the Clifford multiplication. For instance, in case both $m_1$ and $m_2$ are odd, let $\xi\in V$, $\zeta\in W$, $\va\in\mbs_{m_2}$
and $\psi=\psi_1\op\psi_2\in \mbs_{m_1}\op\mbs_{m_1}$, we can use the same expression of \eqref{clifford multiplication} but replace the previous definition of $\xi\cdot\psi$ with a new one $\xi\cdot \psi=(-\xi\cdot \psi_2)\op(-\xi\cdot\psi_1)$
\ (for a close reference, we refer \cite{Bar1998}). In this setting, the half-spinor modules of $\mbs_{m_1+m_2}$ are
\[
\mbs_{m_1+m_2}^+=(\mbs_{m_1}\op\{0\})\ot\mbs_{m_2},
\]
\[
\mbs_{m_1+m_2}^-=(\{0\}\op\mbs_{m_1})\ot\mbs_{m_2}.
\]
It would be better to understand that such changes in product formula give us equivalent
definitions of Clifford multiplications. Indeed, due to the uniqueness of
$\clf(V\op W)$, any definition of the Clifford
multiplication on $\mbs_{m_1+m_2}$ can be identified with \eqref{clifford multiplication}
via a vector space isomorphism.
\end{Rem}

\medskip

Next, let us consider the manifold setting.
Let $(M_1,\ig^{(1)})$ and $(M_2,\ig^{(2)})$ be two oriented Riemannian manifolds of dimensions $m_1$ and $m_2$, respectively. We henceforth suppose that both manifolds are
equipped with a fixed spin structure (for details about spin structures, we refer to
\cite{Friedrich, Lawson} or to the well written self-contained introduction \cite{Hij99}).
This induces a unique spin structure on the
Riemannian product $(N=M_1\times M_2,\ig=\ig^{(1)}\op\ig^{(2)})$.
Indeed, let $\pi_{M_1}$ and $\pi_{M_2}$ denote the projections
on $M_1$ and $M_2$, the tangent bundle of $N$ can be decomposed as
\[
TN=\pi_{M_1}^*TM_1\op \pi_{M_2}^*TM_2.
\]
For simplicity, we omit the projections and write $TN=TM_1\op TM_2$.
And such splitting is orthogonal with respect to $\ig$. Hence
the frame bundle of $N$ can be reduced to a $SO(m_1)\times SO(m_2)$-principal
bundle, and this is isomorphic to the product of the frame bundles over $M_1$ and $M_2$.

\subsection{The Dirac operator}
Fix the spin structures $\sa_{M_1}$ and $\sa_{M_2}$,
let us consider the Clifford bundles (with Clifford multiplications)
$(\C l(M_1),\cdot_{\mbox{\tiny $\ig^{(1)}$}})$, $(\C l(M_2),\cdot_{\mbox{\tiny $\ig^{(2)}$}})$ and spinor bundles
$\mbs(M_1)$, $\mbs(M_2)$ over $M_1$ and $M_2$ respectively.
From the previous considerations in the algebraic settings, we know for the spinor
bundles that
\[
\mbs(N)=\left\{
\aligned
&(\mbs(M_1)\op\mbs(M_1))\ot \mbs(M_2) &\quad & \text{both } m_1 \text{ and } m_2 \text{ are odd}, \\
&\qquad \mbs(M_1)\ot\mbs(M_2) &\quad & m_1 \text{ is even}.
\endaligned \right.
\]
For $X\in TM_1$, $Y\in TM_2$, $\va\in\Ga(\mbs(M_2))$ and
$\psi=\psi_1\op\psi_2\in \Ga(\mbs(M_1)\op\mbs(M_1))$ for both
$m_1$ and $m_2$ odd and $\psi\in\Ga(\mbs(M_1))$ for $m_1$ even, we have
\begin{\equ}\label{product spinor representation}
(X\op Y)\cdot_{\mbox{\tiny $\ig$}} (\psi\ot\va)=(X\cdot_{\mbox{\tiny $\ig^{(1)}$}}\psi)\ot\va+(\om_\C^{M_1}\cdot_{\mbox{\tiny $\ig^{(1)}$}}\psi)\ot(Y\cdot_{\mbox{\tiny $\ig^{(2)}$}}\va)
\end{\equ}
where in case $m_1$ and $m_2$ odd we set $X\cdot_{\mbox{\tiny $\ig^{(1)}$}}\psi=(X\cdot_{\mbox{\tiny $\ig^{(1)}$}}\psi_1)\op(-X\cdot_{\mbox{\tiny $\ig^{(1)}$}}\psi_2)$ and $\om_\C^{M_1}\cdot_{\mbox{\tiny $\ig^{(1)}$}}\psi=i(\psi_2\op-\psi_1)$.

Let $\nabla^{\mbs(M_1)}$ and $\nabla^{\mbs(M_2)}$ be the
Levi-Civita connections on $\mbs(M_1)$ and $\mbs(M_2)$. By
\[
\nabla^{\mbs(M_1)\ot\mbs(M_2)}=\nabla^{\mbs(M_1)}\ot\id_{\mbs(M_2)}
+\id_{\mbs(M_1)}\ot \nabla^{\mbs(M_2)}
\]
we mean the tensor product connection on $\mbs(M_1)\ot\mbs(M_2)$.
If we take $\{X_1,\dots, X_{m_1}\}$ a locally positively oriented orthonormal
frame of $(M_1,\ig^{(1)})$, then the Dirac operator on $M_1$ is (locally) defined by
$D_{\ig^{(1)}}^{M_1}=\sum_{j=1}^{m_1}X_j\cdot_{\mbox{\tiny $\ig^{(1)}$}}\nabla^{\mbs(M_1)}_{X_j}$. Similarly,
if we take $\{Y_1,\dots,Y_{m_2}\}$ a locally positively oriented
orthonormal frame of $(M_2,\ig^{(2)})$, we have
$D_{\ig^{(2)}}^{M_2}=\sum_{j=1}^{m_2}Y_j\cdot_{\mbox{\tiny $\ig^{(2)}$}}\nabla^{\mbs(M_2)}_{Y_j}$.
Evidently, in the product setting,
$\{X_1\op0,\dots, X_{m_1}\op0,0\op Y_1,\dots,0\op Y_{m_2}\}$
 is a local section of the frame bundle of $N$.
Hence formula \eqref{product spinor representation} yields
\[
\aligned
D_{\ig}^N&:=\sum_{j=1}^{m_1}(X_j\op0)\cdot_{\mbox{\tiny $\ig^{(1)}$}}
\nabla^{\mbs(M_1)\ot\mbs(M_2)}_{X_j\op0}+\sum_{j=1}^{m_2}
(0\op Y_j)\cdot_{\mbox{\tiny $\ig^{(2)}$}}
\nabla^{\mbs(M_1)\ot\mbs(M_2)}_{0\op Y_j}  \\
&= \tilde D_{\ig^{(1)}}^{M_1} \ot \id_{\mbs(M_2)}+ (\om_\C^{M_1}\cdot_{\mbox{\tiny $\ig^{(1)}$}}\id_{\mbs(M_1)})\ot D_{\ig^{(2)}}^{M_2}
\endaligned
\]
which defines the Dirac operator on $N=M_1\times M_2$, where
$\tilde D_{\ig^{(1)}}^{M_1}= D_{\ig^{(1)}}^{M_1}\op -D_{\ig^{(1)}}^{M_1}$ if both $m_1$ and $m_2$ are odd
and $\tilde D_{\ig^{(1)}}^{M_1}=D_{\ig^{(1)}}^{M_1}$ if $m_1$ is even.

For the case $m_1+m_2$ even, we have the decomposition
$\mbs(N)=\mbs(N)^+\op\mbs(N)^-$ and, moreover, when restrict $D_{\ig}^N$
on those half-spinor spaces we get
$D_{\ig}^N:\Ga(\mbs(N)^\pm)\to\Ga(\mbs(N)^\mp)$.

\subsection{Analysis on a product conformal structure}\label{analysis case1}

In  this section, we will consider our problem in details and we start with
the case $N=M_1\times M_2$, $m_1=\dim M_1\geq2$ and $m_2=\dim M_2\geq1$.
The case $m_1=m_2=1$, which corresponds to $N=S^1\times S^1$,
will be discussed in Section \ref{S1-bundles}.
From now on, in order to give unified expressions in odd and even cases,
we will write it simply $\mbs(N)=\tilde\mbs(M_1)\ot\mbs(M_2)$ with
\[
\tilde\mbs(M_1)=\left\{
\aligned
&\mbs(M_1)\op\mbs(M_1) &\quad &  m_1 \text{ is odd}, \\
&\qquad \mbs(M_1) &\quad & m_1 \text{ is even}.
\endaligned \right.
\]
and denote $\psi\ot\va$ for a spinor field in $\mbs(N)$
when no confusion can arise.

To have a general view upon the problem,
let us fix a function $\theta: M_1\to (0,+\infty)$ and consider the product conformal metric
$\ig_\ell:=\ell^2\theta^2\ig^{(1)}\op\ig^{(2)}$, where $\ell>0$ is a parameter.
According to the discussions in the previous section, we know for the Dirac
operators that
\[
D_{\ig_\ell}^N= \tilde D_{\ell^2\theta^2\ig^{(1)}}^{M_1} \ot \id_{\mbs(M_2)}+
(\ov{\om_\C^{M_1}}\cdot_{\mbox{\tiny $\ell^2\theta^2\ig^{(1)}$}}\id_{\tilde \mbs(M_1)})\ot D_{\ig^{(2)}}^{M_2}
\]
where $\ov{\om_\C^{M_1}}$ denotes the chirality operator and "$\cdot_{\mbox{\tiny $\ell^2\theta^2\ig^{(1)}$}}$" denotes the Clifford multiplication on $M_1$ with respect to the conformal metric $\ell^2\theta^2\ig^{(1)}$ respectively.

Turning to the nonlinear problems, let us denote $|\cdot|_{\ell^2\theta^2\ig^{(1)}}$ and $|\cdot|_{\ig^{(2)}}$ the natural hermitian metrics on $\mbs(M_1)$ and $\mbs(M_2)$ respectively and $|\cdot|_{\ig_\ell}$ the induced metric on $\mbs(N)$. Set $n=m_1+m_2$ and $n^*=\frac{2n}{n-1}$, we can expand the spinorial Yamabe equation
\[
D_{\ig_\ell}^N\phi=|\phi|_{\ig_\ell}^{n^*-2}\phi, \quad \phi=\bar\psi\ot\va\in\mbs(N)
\]
into
\begin{\equ}\label{equ0}
(\tilde D_{\ell^2\theta^2\ig^{(1)}}^{M_1}\bar\psi )\ot \va+
(\ov{\om_\C^{M_1}}\cdot_{\mbox{\tiny $\ell^2\theta^2\ig^{(1)}$}}\bar\psi)\ot (D_{\ig^{(2)}}^{M_2}\va)=
\big( |\bar\psi|_{\ell^2\theta^2\ig^{(1)}}|\va|_{\ig^{(2)}} \big)^{n^*-2}\bar\psi\ot\va.
\end{\equ}

We will now show how to dispense with the assumption on $(M_2,\ig^{(2)},\sa_{M_2})$. In fact, if $M_2$
possesses a nontrivial eigenspinor $\va_{M_2}$ of constant length for some $\lm\neq0$, then
by substituting $\bar\psi\ot\va_{M_2}$ into \eqref{equ0} we get an equivalent problem
\begin{\equ}\label{equ1}
\tilde D_{\ell^2\theta^2\ig^{(1)}}^{M_1}\bar\psi + \lm\ov{\om_\C^{M_1}}\cdot_{\mbox{\tiny $\ell^2\theta^2\ig^{(1)}$}}\bar\psi=
\big(|\bar\psi|_{\ell^2\theta^2\ig^{(1)}}\big)^{n^*-2}\bar\psi
\end{\equ}
which is sitting on $M_1$.
Here, we adopt the convention that $\lm>0$ since (up to a change
of orientation on $M_1$) the proof for $\lm<0$ is exactly the same.

The following transformation formula describes how Dirac operators for conformally
equivalent metrics are related (see \cite{Hij86, Hit74}).
\begin{Prop}\label{conformal formula}
Let $\ig_0$ and $\ig=f^2\ig_0$ be two conformal metrics on
a Riemannian spin $m$-manifold $M$. Then, there exists an
isomorphism of vector bundles $F:\, \mbs(M,\ig_0)\to
\mbs(M,\ig)$ which is a fiberwise isometry such that
\[
D_\ig^M\big( F(\psi) \big)=F\big( f^{-\frac{m+1}2}D_{\ig_0}^M
\big( f^{\frac{m-1}2}\psi \big)\big).
\]
\end{Prop}

As a direct consequence, the equation \eqref{equ1} can be conformally transformed into
\[
\tilde D_{\ig^{(1)}}^{M_1}\psi+\lm\ell\theta\om_\C^{M_1}\cdot_{\mbox{\tiny $\ig^{(1)}$}}\psi=(\ell\theta)^{m_1-\frac{m_1-1}2n^*}
|\psi|_{\ig^{(1)}}^{n^*-2}\psi.
\]
Moreover, we can consider the rescaling $\psi\mapsto \ell^{-\frac{m_1-1}2}\psi$ in $\tilde\mbs(M_1)$
and denote $\vr=\ell^{-1}$ so that the above equation is equivalent to
\begin{\equ}\label{equ2}
\vr\tilde D_{\ig^{(1)}}^{M_1}\psi+ \lm\theta\om_\C^{M_1}\cdot_{\mbox{\tiny $\ig^{(1)}$}}\psi=\theta^{m_1-\frac{m_1-1}2n^*}
|\psi|_{\ig^{(1)}}^{n^*-2}\psi \quad \text{on } M_1.
\end{\equ}

\subsection{Functional framework}

Our goal is to find solutions of \eqref{equ2} for varying $\vr>0$. Notice that Eq. \eqref{equ2} is well-defined on $M_1$, and $n^*=\frac{2n}{n-1}<\frac{2m_1}{m_1-1}=m_1^*$. It is not necessary to carry the super- and sub-scripts in $(M_1,\ig^{(1)})$, $\tilde D_{\ig^{(1)}}^{M_1}$ and $\om_\C^{M_1}$ during the proofs, hence in order to simplify the notation, we drop these super- and sub-scripts and to consider the model problem
\begin{\equ}\label{equ3}
\vr\tilde D_\ig\psi+ a \om_\C\cdot\psi=b
|\psi|_\ig^{p-2}\psi
\end{\equ}
on a spin $m$-manifold $(M,\ig,\sa)$, where $a,b:M\to(0,+\infty)$ are functions at least $C^1$ smooth and $2<p<m^*:=\frac{2m}{m-1}$.
Unless otherwise stated, we will also drop the subscript of $|\cdot|_{\ig}$
on $\tilde\mbs(M)$ for notation convenience.

For $q>1$, let us denote $L^q:=L^q(M,\tilde\mbs(M))$ which is defined as
the completion of the space $\Ga_c(\tilde\mbs(M)):=\big\{ \psi\in\Ga(\tilde\mbs(M)):\,
\supp(\psi) \text{ is compact} \big\}$ with respect to the norm
$|\cdot|_q^q:=\int_M|\cdot|^q d\vol_\ig$. Particularly, for $q=2$, we have
$L^2$ is a Hilbert space with inner product
$(\cdot,\cdot)_2=\real\int_M(\cdot,\cdot)d\vol_\ig$.

Let us set $A:=\vr\tilde D_\ig + a \om_\C$, our first point is to study the spectrum
$Spec(A)$ of $A$ in $L^2$ . In fact, there is no difficulty to see that $A$ is
self-adjoint and hence $Spec(A)\subset\R$.

\begin{Lem}
For closed Riemannian spin manifold $(M,\ig)$,
\begin{itemize}
\item[$(1)$] $Spec(A)$ is a closed subset of $\R\setminus\{0\}$ consisting of
an unbounded discrete sequence of eigenvalues;

\item[$(2)$] each eigenspace of $A$ is finite-dimensional and consists of
smooth sections;

\item[$(3)$] the eigenspaces of $A$ form a complete orthonormal decomposition
of $L^2$, that is,
\[
L^2=\ov{\bigoplus_{\lm\in Spec(A)} \ker(A-\lm)};
\]

\item[$(4)$] the set $Spec(A)$ is symmetric about the origin.
\end{itemize}
\end{Lem}
\begin{proof}
Let us set $a_{min}=\displaystyle\min_M a>0$ and write $a=\hat a + a_{min}$.
Then, we can introduce the operator $\hat A:=\vr\tilde D_\ig + \hat a\om_\C$
and deduce
\[
(A^2\psi,\psi)_2=(\hat A\psi,\hat A\psi)_2+ a_{min}^2|\psi|_2^2
+2 a_{min}(\hat a \psi,\psi)_2\geq a_{min}^2|\psi|_2^2.
\]
This suggests $Spec(A)\subset (-\infty, -a_{min}]\cup[a_{min},+\infty)$.

Notice that $\om_\C\cdot$ and $\tilde D_\ig$ anticommute, thus
when denote by $A_{min}:=\vr\tilde D_\ig+a_{min}\om_\C$
we get
\[
A_{min}^2=\vr^2\tilde D_\ig^2+ a_{min}^2.
\]
And this implies together with the spectral theorem of Dirac operators
that the $L^2$-spectrum of $A_{min}$
is given by an unbounded discrete sequence of eigenvalues of finite multiplicity.
Since the linear map $\psi\mapsto \hat a\om_\C\cdot\psi$ is relatively compact
with respect to $A_{min}$, we deduce that $A$ has compact resolvent. This proves
$(1)$.  And immediately, the statements $(2)$ and $(3)$ follow from the classical
spectral theory of elliptic self-adjoint operators.

As for the symmetry of $Spec(A)$ about $0$,
it straightforwardly follows from the splitting of the spinor bundle in the
following sense.
In case $m$ is odd, let $\psi=\psi_1\op\psi_2$ be an eigenspinor to an
eigenvalue $\lm\in Spec(A)$, we have $\psi=\psi^s+\psi^d$ with
\[
\psi^s=\frac{\psi_1+\psi_2}2\op\frac{\psi_1+\psi_2}2 \quad
\text{and} \quad
 \psi^d=\frac{\psi_1-\psi_2}2\op-\frac{\psi_1-\psi_2}2.
\]
Since $A=\vr\tilde D_\ig+a\om_\C$
exchanges the above splitting of $\mbs(M)\op\mbs(M)$, we have
\[
\vr\tilde D_\ig\psi^s+a\om_\C\cdot\psi^s=\lm \psi^d \quad
\text{and} \quad
\vr\tilde D_\ig\psi^d+a\om_\C\cdot\psi^d=\lm \psi^s
\]
and therefore
\[
A(\psi^s-\psi^d)=-\lm(\psi^s-\psi^d).
\]
In case $m$ even, $\mbs(M)$ itself splits into $\mbs^+(M)\op\mbs^-(M)$
of eigenspaces of $\om_\C$ and
there is a representation of the Dirac operator as
\[
\tilde D_\ig=D_\ig=\begin{pmatrix}
0 & D_\ig\big|_{\mbs^-(M)} \\
D_\ig\big|_{\mbs^+(M)} & 0
\end{pmatrix}: \Ga(\mbs^+(M))\op\Ga(\mbs^-(M)) \to
\Ga(\mbs^-(M))\op\Ga(\mbs^+(M)).
\]
One can pass from $\mbs^+(M)$ to $\mbs^-(M)$ by taking the same
underlying vector bundle $\tilde\mbs_M=\mbs^+(M)=\mbs^-(M)$
and set
\[
\tilde\mbs^1_M:=\big\{ (\psi,-i\psi):\, \psi\in\tilde\mbs_M \big\} \quad
\text{and} \quad
\tilde\mbs^2_M:=\big\{ (\psi,i\psi):\, \psi\in\tilde\mbs_M \big\}.
\]
Then there is an isomorphism between vector bundles
\[
\tilde\mbs_M\op\tilde\mbs_M\to \tilde\mbs^1_M\op\tilde\mbs^2_M, \quad
(\psi,0)   \mapsto   \frac1{\sqrt2}
(\psi, -i\psi)  \quad \text{and} \quad
(0,\psi)   \mapsto   \frac1{\sqrt2}
(\psi, i\psi)
\]
so that the operator $A$ has the representation
\[
\tilde D_\ig \mapsto \begin{pmatrix}
D_\ig\big|_{\tilde\mbs} & 0  \\
0 & -D_\ig\big|_{\tilde\mbs}
\end{pmatrix} \quad \text{and} \quad
\om_\C \mapsto i \begin{pmatrix}
0 & \id_{\tilde\mbs}\\
-\id_{\tilde\mbs} & 0
\end{pmatrix}
\]
which is exactly the same as the case $m$ odd.
Therefore, one easily checks $Spec(A)$ is symmetric about $0$.
\end{proof}

From now on, we make the assumption that $(M,\ig)$ is closed.
Then we can choose a complete orthonormal basis $\psi_{\pm1},\psi_{\pm2},\dots$
of $L^2$ consisting of the eigenspinors of $A$, i.e.
$A\psi_{\pm k}=\lm_{\pm k}\psi_{\pm k}$ and the spectrum
$Spec(A)$ will be denoted as
\[
\cdots\leq \lm_{-2}\leq \lm_{-1}<0< \lm_1\leq\lm_2\leq\cdots,
\]
where each eigenvalue appears with its multiplicity.
Particularly, we have $\lm_k=-\lm_{-k}$ and $|\lm_{\pm k}|\to+\infty$ as $k\to\infty$.

From a variational point of view, to study Eq. \eqref{equ3}, we need to
define the unbounded operator $|A|^s:L^2\to L^2$, $s\geq0$, by
\[
|A|^s \psi = \sum_{k=-\infty}^{\infty}|\lm_k|^s \al_k\psi_k
\]
where $\psi=\sum_{k=-\infty}^\infty\al_k\psi_k\in L^2$. In this way, we can
introduce the domain of $|A|^s$ in $L^2$ as
\[
\msh^s:=\bigg\{ \psi=\sum_{k=1}^\infty\al_k\psi_k\in L^2:\,
\sum_{k=-\infty}^\infty |\lm_k|^{2s} |\al_k|^2<\infty \bigg\}.
\]
It is worth pointing out that $\msh^{\frac12}$ coincides with the Sobolev space of order
$\frac12$, that is  $W^{\frac12,2}(M,\tilde\mbs(M))$
(see for instance \cite{Adams, Ammann}).
Moreover, we can equip $\ch:=\msh^{\frac12}$ an inner product
\begin{\equ}\label{the norm}
\inp{\psi}{\va}_\vr:=\frac1{\vr^m}\,\real\int_M\big(|A|^{1/2}\psi,
|A|^{1/2}\va\big) d\vol_\ig
\end{\equ}
and the induced norm $\|\cdot\|_\vr$
such that $(\ch,\inp{\cdot}{\cdot}_\vr)$ becomes a Hilbert space.
Remark that, in the above notations, we have emphasized the dependence
on the parameter $\vr$ because it is already hidden in the operator $A$ and its spectrum.
The dual space of $\ch$ will be denoted by $\ch^*=W^{-\frac12,2}(M,\tilde\mbs(M))$. Identifying $\ch$ with $\ch^*$
we have $\inp{\cdot}{\cdot}_\vr$ can be used to denote the norm on $\ch^*$.

On the Banach space $L^q$, $q>1$, we equip it a new norm
\[
|\psi|_{q,\vr}=\bigg( \frac1{\vr^m}\int_M|\psi|^q d\vol_\ig \bigg)^{\frac1q}.
\]
Then, recall $m^*=\frac{2m}{m-1}$, we get

\begin{Lem}\label{embeding lemma}
If $\vr>0$ is small, then for any $q\in[2,m^*]$ the embedding
$\id_\ch:(\ch,\|\cdot\|_\vr)\hookrightarrow (L^q, |\cdot|_{q,\vr})$
is a continuous map independent of $\vr$, that is, there exists $c_q>0$
does not depend on $\vr$ such that
\[
|\psi|_{q,\vr}\leq c_q \|\psi\|_\vr \quad \text{for all } \psi\in\ch.
\]
In particular, the embedding is compact for $q\in[2,m^*)$.
\end{Lem}
\begin{proof}
Our strategy is to use interpolation inequalities and we sketch the proof as
follows. To begin with, we observe
that, for $\psi\in\Ga(\tilde\mbs(M))$,
\[
A^2\psi=\vr^2(\tilde D_\ig)^2\psi + a^2\psi +\vr\nabla a\cdot \om_\C
\cdot \psi.
\]
When denoted by $Scal_\ig$ the scalar curvature of $(M,\ig)$, by
Schr\"odinger-Lichnerowicz formula, we have
\[
(\tilde D_\ig)^2=\nabla^*\nabla+\frac14 Scal_\ig
\]
where $\nabla^*\nabla$ is the standard connection Laplacian.
Hence we get
\begin{eqnarray}\label{norm esti1}
\big||A|\psi\big|_2^2&=&|A\psi|_2^2 = \int_M(A^2\psi,\psi)d\vol_\ig  \nonumber\\
&=&\int_M \vr^2|\nabla\psi|^2
+\Big(a^2+\frac{\vr^2}4Scal_\ig\Big)|\psi|^2
+ \vr(\nabla a\cdot\om_\C\cdot\psi,\psi) d\vol_\ig
\end{eqnarray}
for all $\psi\in\Ga(\tilde\mbs(M))$.

Since $M$ is closed, the curvature function $Scal_\ig$ is bounded. Noting that
$a>0$ is of class $C^1$, we assert from \eqref{norm esti1} that,
for small values of $\vr$, the space $\msh^1$ coincides with the standard
Sobolev space $W^{1,2}(M,\tilde\mbs(M))$ and,
for each fixed $\vr$, the map $\psi\mapsto\big| |A|\psi\big|_2$
defines an equivalent norm on $W^{1,2}(M,\tilde\mbs(M))$.

Recalling the classical Sobolev embedding theorems, we can conclude that there exists
a positive constant $C$ (which depends on the dimension) such that
\[
\Big( \int_M |\psi|^{\frac{2m}{m-2}}d\vol_\ig \Big)^{\frac{m-2}m} \leq
C \int_M\big(|\nabla\psi|^2 +  |\psi|^2 \big)d\vol_\ig
\]
for all $\psi\in W^{1,2}(M,\tilde\mbs(M))$. One easily gets
\[
\Big( \frac1{\vr^m}\int_M |\psi|^{\frac{2m}{m-2}}d\vol_\ig \Big)^{\frac{m-2}m} \leq
\frac{C}{\vr^m} \int_M\big(\vr^2|\nabla\psi|^2  + |\psi|^2 \big)d\vol_\ig
\]
for all $\psi\in \msh^1$ provided that $\vr$ is small.

In the next, we use the following notation for a different norm
of $\psi\in W^{1,2}(M,\tilde\mbs(M))$
\[
\|\psi\|_{1,2;\vr}^2:=\frac1{\vr^m} \int_M\big(\vr^2|\nabla\psi|^2  + |\psi|^2 \big)d\vol_\ig.
\]
Then, consider the interpolation couples
\[
(\msh^1,\|\cdot\|_{1,2;\vr})\hookrightarrow \big(L^{m^*},|\cdot|_{m^*,\vr}\big)
\]
and
\[
(\msh^0,|\cdot|_{2,\vr}) \hookrightarrow (L^2,|\cdot|_{2,\vr}),
\]
we can easily assert from the Calder\'on-Lions interpolation theorem \cite{RS}
that the embedding constant for $(\msh^{\frac12},\|\cdot\|_\vr)\hookrightarrow
(L^{m^*},|\cdot|_{m^*,\vr})$ is independent of $\vr$.

For the compactness, we only need to point out that the embedding
$W^{1,2}(M,\tilde\mbs(M))\hookrightarrow L^q$ is compact for
$q\in[2,m^*)$. Therefore, by the interpolation theorem again,
we have $\ch=\msh^{\frac12}\hookrightarrow L^q$ is compact
for $q\in[2,m^*)$, which completes the proof.
\end{proof}

\begin{Rem}
Lemma \ref{embeding lemma} gives rise to a question on the
$\vr$-dependence of the embedding constant $c_q$ for large values
of $\vr$. This is quite involved. 
Since the form domain of the operator $A$ is the fractional
Sobolev space $W^{\frac12,2}(M,\tilde\mbs(M))$, the embedding
$\ch \hookrightarrow L^q$, $q\in[2,m^*]$ exists in any circumstance.
However, from the definition of the norm $\|\cdot\|_\vr$ in \eqref{the norm},
one can not get the explicit dependence on the parameter $\vr$. Moreover,
by the Schr\"odinger-Lichnerowicz formula and formula \eqref{norm esti1}, we see that,
for large $\vr$, the interaction between the function $a$ and the scalar curvature
$Scal_\ig$ enters into play. In this situation, the embedding constant for
$W^{1,2}(M,\tilde\mbs(M))\hookrightarrow L^{m^*}$ is
$\vr$-dependent when $Scal_\ig$ possesses certain negative parts.
This fact will probably impact the embedding constant of $\ch$ into $L^{m^*}$.
\end{Rem}

\medskip

Recall that we have an $(\cdot,\cdot)_2$-orthogonal decomposition
\[
L^2=L_\vr^+\op L_\vr^-, \quad \psi=\psi^++\psi^-
\]
with
\[
L_\vr^+:= \ov{\bigoplus_{k=1}^{+\infty} \ker(A-\lm_k)} \quad \text{and} \quad
L_\vr^-:= \ov{\bigoplus_{k=-1}^{-\infty} \ker(A-\lm_k)}
\]
so that $A$ is positive definite on $L_\vr^+$ and negative definite on $L_\vr^-$.
Then, this leads to the orthogonal decomposition of $\ch$ with respect to
the inner product $\inp{\cdot}{\cdot}_\vr$ as
\[
\ch=\ch_\vr^+\op \ch_\vr^-, \quad \ch_\vr^\pm=\ch\cap L_\vr^\pm.
\]

With the above notations, we have Eq. \eqref{equ3} is the Euler-Lagrange equation
of the functional
\begin{eqnarray}\label{the functional}
\cl_\vr(\psi)&=&\frac1{\vr^m}\int_M\Big( \frac12(A\psi,\psi)-\frac{b}{p}|\psi|^{p}
\Big) d\vol_\ig  \nonumber\\[0.5em]
&=&\frac12\big(\|\psi^+\|_\vr^2-\|\psi^-\|_\vr^2 \big)-
\frac1{\vr^m\, p}\int_M b|\psi|^{p}d\vol_\ig
\end{eqnarray}
defined on $\ch=\ch_\vr^+\op\ch_\vr^-$. And by Lemma \ref{embeding lemma}, we have
$\cl_\vr\in C^2(\ch,\R)$.

We emphasize that, by abuse of notation, we just simply write
$\psi=\psi^++\psi^-$  for the orthogonal decomposition of $\ch$ without
mention its dependence on $\vr$. However, one should always keep in mind that,
for different values of $\vr$, such decomposition of a spinor $\psi$ is different.

\section{Existence of solutions}

We shall now investigate the existence of a nontrivial solution for
Eq. \eqref{equ3}. This is equivalent to find nontrivial critical points
of the functional $\cl_\vr$ in the Hilbert space $\ch$.
Fortunately, this is not a difficult task. Indeed, by noting that
$p<m^*$,
the compact embedding $\ch\hookrightarrow L^{p}$ sheds light on
several ways to obtain the existence issue.

A first approach is to construct, on a subspace of $\ch$,
a functional having a mountain pass geometry whose critical points
are in one-to-one correspondence with critical points of $\cl_\vr$.
This idea can be found in a paper of Buffoni, Jeanjean and
Stuart in 1993 where the authors studied the solutions to the Choquard-Pekar equation
in $\R^3$, see \cite{BJS}. Compared with the problem \eqref{equ3},
the strategy is to use a global Lyapunov-Schmidt
reduction to control the part of the solutions in the space $\ch_\vr^-$. This will lead
to study a functional defined only on $\ch_\vr^+$. Then a direct application of the
Mountain Pass Theorem gives the existence of a critical point. Such reduction
argument essentially requires the super-quadratic part to be
strictly convex and  $C^2$ smooth.
This is automatically satisfied in our situation since the the super-quadratic part in $\cl_\vr$
more or less behaves like the $L^{p}$-integral.

This approach was extended by several papers to study different nonlinear PDE
problems. Typical (but not comprehensive) results can be found in \cite{Ackermann}
where Schr\"odinger equations with periodic potentials and in \cite{MH, MY}
where semiclassical Hamiltonian elliptic systems were studied.

A second approach devotes in searching directly critical points of indefinite functionals.
Benci and Rabinowitz \cite{BR1979} opened this route. They constructed deformations
having special representations by solving appropriate differential equations approximately
by time discretization. This approach was subsequently improved by Hofer in \cite{Hofer1983}.
Several applications are exhibited in the study of periodic solutions of the one-dimensional wave
equation, in the study of periodic solutions of Hamiltonian systems of ordinary differential equations,
in the existence theory of systems of elliptic equations, in the study of resonance problems of the
Landesman and Lazer type, etc.

Substantial improvements along this approach were made by
Kryszewski and Szulkin \cite{KS}. The authors built an abstract linking theorem to
obtain critical points of a strongly indefinite functional.
And applications of this linking theorem give rise to the existence of a solution
to nonlinear Schr\"odinger equations with periodic potentials and some general nonlinearities.
The convexity of the super-quadratic part is not required and the energy functional is
of $C^1$. This approach was subsequently refined in the work of
Bartsch and Ding \cite{BD2006} via a general setting on Banach spaces.

The third approach (using the fundamental idea in calculus of variations)
is to consider a constrained variational structure. Specifically, to our model problem
\eqref{equ3}, the procedure is to construct, on the unit sphere of $\ch_\vr^+$, a map $\chi_\vr$
in $\ch$  such that the composition $\cl_\vr\circ\chi_\vr$ is of $C^1$ smooth
and all its critical points on the unit sphere
correspond to solutions of \eqref{equ3}. This formulation
was developed by Szulkin and Weth \cite{Szulkin-Weth:2010}, and this idea can
be viewed as a refinement of the first approach in the sense that $\chi_\vr$
can be understand as a normalized reduction map of $\cl_\vr$ to $\ch_\vr^+$.
However, being differently, the advantage of this approach
is to remove the convexity and to give the most simplified characterization
of the critical value.

Here, owning to the original face of the problem \eqref{equ3},
we adopt the first approach mentioned above to obtain the existence results.
The basic reason, which tempts us into choosing this approach,
is that we shall use the $C^2$ property and the reduction procedure to go
further to get some useful estimate on the critical levels so that we could draw
more information on our geometric objects.

Let's begin with the following compactness result of the functional $\cl_\vr$. Since the proof is classical, we simply omit it here.

\begin{Lem}\label{PS-condition}
For each $\vr>0$ small, $\cl_\vr$ satisfies the $(P.S.)_c$-condition
for $c\geq0$, that is,
\[ \left.
\aligned
\cl_\vr(\psi_n) \to c \  \\
\cl_\vr'(\psi_n)\to0 \
\endaligned \right\} \Rightarrow \{\psi_n\} \text{ possesses
a convergent subsequence in } \ch.
\]
Moreover, $\psi_n\to0$ if and only if $c=0$.
\end{Lem}
%

Now, for the functionals $\cl_\vr$, we have
\begin{Prop}\label{reduction1}
For each $\vr>0$ small,
\begin{itemize}
\item[$(1)$] there exists $g_\vr\in C^1(\ch_\vr^+,\ch_\vr^-)$ such that
\[
\forall w\in\ch_\vr^-, \  w\neq g_\vr(u) \Rightarrow
\cl_\vr(u+w)<\cl_\vr(u+g_\vr(u)),
\]
in particular,
\[
\|g_\vr(u)\|_\vr^2\leq \frac2{\vr^m\,p}\int_M b|u|^{p}d\vol_\ig
\]
and $\cl_\vr'(u+g_\vr(u))[w]\equiv 0$ for all $w\in\ch_\vr^-$;

\item[$(2)$] denoted by
\[
I_\vr: \ch_\vr^+\to \R, \quad I_\vr(u)=\cl_\vr(u+g_\vr(u)),
\]
if $\{u_n\}$ is a $(P.S.)$-sequence for $I_\vr$ then $\{u_n+g_\vr(u_n)\}$ is a
$(P.S.)$-sequence for $\cl_\vr$; 

\item[$(3)$] there exists $\psi_\vr\in\ch$ such that $\cl_\vr'(\psi_\vr)=0$ and $\psi_\vr\neq0$.
\end{itemize}
\end{Prop}
\begin{proof}
By Lemma \ref{PS-condition}, we may apply the arguments in \cite[Section 2]{BJS} to get
the existence of one critical point for $I_\vr$ of Mountain Pass type.
\end{proof}

Next we are intend to give a characterization of the critical point $\psi_\vr$ obtained
in the above proposition.

\begin{Lem}\label{reduction2}
For every $u\in \ch_\vr^+\setminus\{0\}$, the map $I_{\vr,u}:\R\to\R$,
$I_{\vr,u}(t)=I_\vr(tu)$, is of class $C^2$ and satisfies
\[
I_{\vr,u}'(t)=0,\ t>0 \quad \Longrightarrow \quad I_{\vr,u}''(t)<0.
\]
Moreover $I_{\vr,u}(0)=I_{\vr,u}'(0)=0$, $I_{\vr,u}''(0)>0$.
\end{Lem}
\begin{proof}
In order to see this, we compute  $I_{\vr,u}'(t)=\cl_\vr'(tu+g_\vr(tu))[u]$ which
suggests that $I_{\vr,u}$ is $C^2$. As we can see, the implication is equivalent to:
\begin{\equ}\label{I-equivalent}
I_\vr'(u)[u]=0, \ u\neq0 \quad \Longrightarrow \quad I_\vr''(u)[u,u]<0.
\end{\equ}
For simplicity, let us denote $\Psi_\vr:\ch\to\R$ by $\Psi_\vr(\psi)
=\frac1{\vr^m\,p}\int_M b |\psi|^{p}d\vol_\ig$ and set
$\psi=u+g_\vr(u)$ and $\chi=g_\vr'(u)[u]-g_\vr(u)$. By using
$\cl_\vr'(u+g_\vr(u))|_{\ch_\vr^-}\equiv0$, we have \eqref{I-equivalent}
is a consequence of the following computation:
\begin{eqnarray}\label{eee1}
I_\vr''(u)[u,u]&=&\cl_\vr''(\psi)[u+g_\vr'(u)[u],u]=\cl_\vr''(\psi)[\psi+\chi,\psi+\chi]  \nonumber\\
&=&\cl_\vr''(\psi)[\psi,\psi]+2\cl_\vr''(\psi)[\psi,\chi]
+\cl_\vr''(\psi)[\chi,\chi]  \nonumber\\
&=&I_\vr'(u)[u]+\big(\Psi_\vr'(\psi)[\psi]-\Psi_\vr''(\psi)[\psi,\psi]\big)
+2\big(\Psi_\vr'(\psi)[\chi]-\Psi_\vr''(\psi)[\psi,\chi] \big) \nonumber\\
& &\quad -\Psi_\vr''(\psi)[\chi,\chi]
-\|\chi\|_\vr^2 \nonumber\\
& \leq& I_\vr'(u)[u]-\frac1{\vr^m}\frac{p-2}{p-1}\int_M b|\psi|^{p}
 d\vol_\ig-\|\chi\|_\vr^2.
\end{eqnarray}
\end{proof}

A natural constraint for $I_\vr$ is to consider the associated Nehari manifold:
\[
\msn_\vr:=\big\{ u\in\ch_\vr^+\setminus\{0\}:\, I_\vr'(u)[u]=0 \big\}.
\]
By Lemma \ref{reduction2} this is a smooth submanifold of codimension $1$
in $\ch_\vr^+$. And consequently, the critical point found in Proposition \ref{reduction1} $(3)$
can be characterized by
\begin{\equ}\label{ga-vr}
\ga_\vr:=\cl_\vr(\psi_\vr)=\inf_{u\in\ch_\vr^+\setminus\{0\}}
\max_{\psi\in \R u\op \ch_\vr^-}\cl_\vr(\psi)
=\inf_{u\in\ch_\vr^+\setminus\{0\}}\max_{t>0} I_\vr(tu)
=\inf_{u\in\msn_\vr}I_\vr(u).
\end{\equ}
For later purpose, it is worth to point out that, by Lemma \ref{embeding lemma},
there holds
\begin{\equ}\label{ga-vr geq tau}
I_\vr(tu)
\geq\frac{t^2}2\|u\|_\vr^2-\frac{c_{p}^{p}\, t^{p}}{p}\max b\, \|u\|_\vr^{p}
\qquad \forall u\in\ch_\vr^+\setminus\{0\}, \ \forall t>0.
\end{\equ}
Hence there exists $\tau_0>0$ independent of $\vr$ such that $\ga_\vr\geq\tau_0$.

In what follows, we intend to pass to the limit $\vr\to0$ and consider the
convergence of the min-max level $\ga_\vr$. The idea is to use certain test spinors
in the functional $\cl_\vr$. For this purpose, first of
all, we need to establish an upper bound estimate. Without loss of generality,
we assume that $\{\phi_\vr\}\subset\ch$ is an arbitrary sequence such that
\begin{\equ}\label{assumption0}
c_1\leq \cl_\vr(\phi_\vr)\leq c_2 \quad \text{and} \quad
\|\cl_\vr'(\phi_\vr)\|_\vr\to0
\end{\equ}
as $\vr\to0$ for some constants $c_1,c_2>0$. Here, we have
identified the dual space $\ch^*$ with $\ch$.

\begin{Lem}\label{a1}
Under \eqref{assumption0}, we have
\begin{itemize}
\item[$(1)$] $\|\phi_\vr\|_\vr$ is uniformly bounded in $\vr$;

\item[$(2)$] $\|\phi_\vr^- - g_\vr(\phi_\vr^+)\|_\vr\leq O\big( \|\cl_\vr'(\phi_\vr)\|_\vr \big)$
    as $ \vr\to0$;

\item[$(3)$] $I_\vr'(\phi_\vr^+)\to0$ as $\vr\to0$ in the dual space of $\ch_\vr^+$.
\end{itemize}
\end{Lem}
\begin{proof}
For the boundedness, we recall that Lemma \ref{embeding lemma} implies
the embedding constant for $\ch\hookrightarrow L^{p^*}$ is independent of
$\vr$, and hence the arguments in Lemma \ref{PS-condition} can be employed.

For $(2)$, let us first set $z_\vr=\phi_\vr^++g_\vr(\phi_\vr^+)$ and
$v_\vr=\phi_\vr^--g_\vr(\phi_\vr^+)$. Then we have $v_\vr\in \ch_\vr^-$
and, by the definition of $g_\vr$,
\[
0=\cl_\vr'(z_\vr)[v_\vr]=-\inp{g_\vr(\phi_\vr^+)}{v_\vr}_\vr-\frac1{\vr^m}\real\int_M
b |z_\vr|^{p-2}(z_\vr,v_\vr)d\vol_\ig.
\]
Since $\|\cl_\vr'(\phi_\vr)\|_\vr\to 0$ as $\vr\to0$, it follows that
\[
o(\|v_\vr\|_\vr)=\cl_\vr'(\phi_\vr)[v_\vr]=-\inp{\phi_\vr^-}{v_\vr}-\frac1{\vr^m}
\real\int_M b |\phi_\vr|^{p-2}(\phi_\vr,v_\vr)d\vol_\ig.
\]
And hence, we get
\begin{\equ}\label{e1}
\aligned
o(\|v_\vr\|_\vr)&=\|v_\vr\|_\vr^2+\frac1{\vr^m}
\real\int_M b |\phi_\vr|^{p-2}(\phi_\vr,v_\vr)d\vol_\ig \\
&\qquad -\frac1{\vr^m}\real\int_M
b |z_\vr|^{p-2}(z_\vr,v_\vr)d\vol_\ig.
\endaligned
\end{\equ}
Remark that the map $\psi\to |\psi|^{p}$ is convex, we have
\[
\frac1{\vr^m}
\real\int_M b |\phi_\vr|^{p-2}(\phi_\vr,v_\vr)d\vol_\ig
-\frac1{\vr^m}\real\int_M
b |z_\vr|^{p-2}(z_\vr,v_\vr)d\vol_\ig\geq 0.
\]
Thus, from \eqref{e1}, we can infer that $\|v_\vr\|_\vr\leq O\big( \|\cl_\vr'(\phi_\vr)\|_\vr \big)$
as $\vr\to0$.

In order to check $(3)$ we compute $I_\vr'(\phi_\vr^+)=\cl_\vr'\big(\phi_\vr^++g_\vr(\phi_\vr^+)\big)$,
which implies $\|I_\vr'(\phi_\vr^+)\|_\vr\to0$ as $\vr\to0$ is a
direct consequence of the $C^2$ smoothness of $\cl_\vr$.
\end{proof}

Next, let us introduce the functional $H_\vr: \ch_\vr^+\to\R$ by $H_\vr(u)=I_\vr'(u)[u]$.
Then, it is clear that $H_\vr$ is $C^1$ and its derivative is given by the formula
\[
H_\vr'(u)[w]=I_\vr'(u)[w]+I_\vr''(u)[u,w]
\]
for $u,w\in\ch_\vr^+$. We also have $\msn_\vr=H_\vr^{-1}(0)\setminus\{0\}$.
Moreover, by \eqref{eee1}, we have
\begin{\equ}\label{eee2}
H_\vr'(u)[u]\leq 2 H_\vr(u)-\frac1{\vr^m}\frac{p-2}{p-1}\int_M b \big|u+g_\vr(u)\big|^{p}d\vol_\ig.
\end{\equ}
for any $u\in \ch_\vr^+$.
\begin{Prop}\label{estimate prop}
For the sequence $\{\phi_\vr\}$ in \eqref{assumption0}, there exists
$\{t_\vr\}\subset\R$ such that $t_\vr\phi_\vr^+\in\msn_\vr$ and
$|t_\vr-1|\leq O\big(\|I_\vr'(\phi_\vr^+)\|_\vr\big)$.
\end{Prop}
\begin{proof}
We begin with the observation: due to the condition \eqref{assumption0}
and Lemma \ref{a1} $(3)$, there holds
\begin{\equ}\label{o1}
\liminf_{\vr\to0}\frac1{\vr^m}\int_M b \big| \phi_\vr^++g_\vr(\phi_\vr^+) \big|^{p}d\vol_\ig
\geq c_0
\end{\equ}
for some constant $c_0>0$. Let us set $\eta_\vr:(0,\infty)\to\R$ by
$\eta_\vr(t)=H_\vr(t\phi_\vr^+)$. One easily checks that
$t\eta_\vr'(t)=H_\vr'(t\phi_\vr^+)[t\phi_\vr^+]$ for all $t>0$. Hence,
by \eqref{eee2} and Taylor's formula, we get
\begin{\equ}\label{e2}
t\eta_\vr'(t)\leq 2 \eta_\vr(1)-\frac1{\vr^m}\frac{p-2}{p-1}\int_M
b \big| \phi_\vr^++g_\vr(\phi_\vr^+) \big|^{p}d\vol_\ig + C|t-1|
\end{\equ}
for $t$ close to $1$ with $C>0$ independent of $\vr$. Here we have used
the uniform boundedness of $\eta_\vr'(t)$ on bounded intervals.

Notice that $\eta_\vr(1)=I_\vr'(\phi_\vr^+)[\phi_\vr^+]\to0$ as $\vr\to0$,
we conclude from \eqref{o1} and \eqref{e2} that there exists a small constant $\de>0$ such that
\[
\eta_\vr'(t)\leq-\de \text{ for all } t\in(1-\de,1+\de) \text{ and } \vr \text{ small enough}.
\]
Moreover, from Lemma \ref{reduction2}, we have $\eta_\vr(1-\de)>0$ and
$\eta_\vr(1+\de)<0$. Then, by Inverse Function Theorem, $t_\vr:=\eta_\vr^{-1}(0)$
exists and
\[
u_\vr:=t_\vr\phi_\vr^+\in\msn_\vr\cap \span\{\phi_\vr^+\}
\]
is well-defined for all $\vr$ small enough. Furthermore, since $|\eta_\vr'(t)^{-1}|$
is bounded by a constant, say $c_\de>0$, on $(1-\de,1+\de)$, we consequently get
\[
\|u_\vr-\phi_\vr^+\|_\vr=
|\eta_\vr^{-1}(0)-\eta_\vr^{-1}(H_\vr(\phi_\vr^+))|\cdot\|\phi_\vr^+\|_\vr
\leq c_\de |H_\vr(\phi_\vr^+)|\cdot\|\phi_\vr^+\|_\vr.
\]
Now the conclusion follows from $H_\vr(\phi_\vr^+)\leq O\big(\|I_\vr'(\phi_\vr^+)\|_\vr\big)$.
\end{proof}

\begin{Cor}\label{key corollary}
For the sequence $\{\phi_\vr\}$ in \eqref{assumption0}, there exists
$\{u_\vr\}$ such that $u_\vr\in\msn_\vr$ and
$\|\phi_\vr-u_\vr-g_\vr(u_\vr)\|_\vr\leq O(\|\cl_\vr'(\phi_\vr)\|_\vr)$.
Particularly,
\[
\max_{t>0}I_\vr(t\phi_\vr^+)=I_\vr(u_\vr)\leq \cl_\vr(\phi_\vr)
+O\big( \|\cl_\vr'(\phi_\vr)\|_\vr^2 \big).
\]
\end{Cor}
\begin{proof}
To see this, let $u_\vr=t_\vr\phi_\vr^+$ be as in Proposition \ref{estimate prop}
and set $z_\vr=\phi_\vr^++g_\vr(\phi_\vr^+)$. Then one obtains from Lemma \ref{a1} that
\begin{\equ}\label{e3}
\aligned
\|\phi_\vr-u_\vr-g_\vr(u_\vr)\|_\vr&\leq \|\phi_\vr-z_\vr\|_\vr +
|t_\vr-1|\cdot\|\phi_\vr^+\|_\vr + \|g_\vr(\phi_\vr^+)-g_\vr(u_\vr)\|_\vr \\
&\leq O\big( \|\cl_\vr'(\phi_\vr)\|_\vr \big)+O\big( \|I_\vr'(\phi_\vr^+)\|_\vr \big)
\endaligned
\end{\equ}
where we have used an easily checked inequality
\[
\|g_\vr(\phi_\vr^+)-g_\vr(u_\vr)\|_\vr\leq \|g_\vr'(\tau \phi_\vr^+)\|_{\ch_\vr^+\to\ch_\vr^-}\cdot
\|\phi_\vr^+-u_\vr\|_\vr= O(|t_\vr-1|)
\]
for some $\tau$ between $t_\vr$ and $1$.
Remark that $I_\vr'(\phi_\vr^+)=\cl_\vr'(z_\vr)$, by using the $C^2$ smoothness
of $\cl_\vr$, we have
\[
\|I_\vr'(\phi_\vr^+)\|_\vr=\|\cl_\vr'(z_\vr)\|_\vr\leq \|\cl_\vr'(\phi_\vr)\|_\vr
+O(\|\phi_\vr-z_\vr\|_\vr)=O(\|\cl_\vr'(\phi_\vr)\|_\vr)
\]
This together with \eqref{e3} implies
\[
\|\phi_\vr-u_\vr-g_\vr(u_\vr)\|_\vr\leq O(\|\cl_\vr'(\phi_\vr)\|_\vr).
\]

Now, by Talyor's formula, we can obtain
\[
\aligned
\cl_\vr(\phi_\vr)&=\cl_\vr(u_\vr+g_\vr(u_\vr))+\cl_\vr'(u_\vr+g_\vr(u_\vr))[\phi_\vr-u_\vr-g_\vr(u_\vr)]
+O\big(\|\cl_\vr'(\phi_\vr)\|_\vr^2\big)  \\
&= I_\vr(u_\vr)+I_\vr'(u_\vr)[\phi_\vr^+-u_\vr]+O\big(\|\cl_\vr'(\phi_\vr)\|_\vr^2\big) .
\endaligned
\]
Notice that $u_\vr=t_\vr\phi_\vr^+\in\msn_\vr$, we have $I_\vr'(u_\vr)[\phi_\vr^+-u_\vr]\equiv0$
and this implies the last estimate.
\end{proof}

As a immediate consequence of Corollary \ref{key corollary}, we can show a explicit
upper bound of $\ga_\vr$ if we find some test spinors $\{\phi_\vr\}$ satisfying
\eqref{assumption0}.

\section{Energy gap for solutions in Euclidean spaces: the bubbles}\label{energy gap}

We consider solutions to the equation
\begin{\equ}\label{limit equ}
\tilde D_{\ig_{\R^m}}\psi + \nu \om_\C\cdot\psi = \ka |\psi|^{p-2}\psi
\quad \text{on } \R^m
\end{\equ}
belonging to the class $W^{\frac12,2}(\R^m,\tilde\mbs(\R^m))$, where $\nu,\ka>0$
are constants,
\[
\tilde\mbs(\R^m)=\left\{
\aligned
&\mbs(\R^m)\op\mbs(\R^m) &\quad &  m \text{ is odd}, \\
&\qquad \mbs(\R^m) &\quad & m \text{ is even},
\endaligned \right.
\]
$\tilde D_{\ig_{\R^m}}=D_{\ig_{\R^m}}\op -D_{\ig_{\R^m}}$
if $m$ is odd and $\tilde D_{\ig_{\R^m}}=D_{\ig_{\R^m}}$ if $m$ is even.
These solutions correspond to "bubbles" or test spinors for our variational problem.

First of all, let us denote $A_\nu=\tilde D_{\ig_{\R^m}}+\nu\om_\C$. By a straightforward
calculation we see that $A_\nu$ is a self-adjoint operator on $L^2$ and has its spectrum
$Spec(A_\nu)=(-\infty,-\nu]\cup[\nu,+\infty)$. Following Amann \cite{Amann}, denote
 $(E_\lm)_{\lm\in\R}$ the spectral resolution of $A_\nu$ and define the
orthogonal projections by
\[
P_\nu=\int_{-\infty}^0 d E_\lm, \quad  Q_\nu=\int_0^{\infty} d E_\lm.
\]
Then the decomposition of $\ce=W^{\frac12,2}(\R^m,\tilde\mbs(\R^m))=\ce_\nu^+\op\ce_\nu^-$ is induced by
\[
\ce_\nu^-=\ce\cap P_\nu (L^2) \quad \text{and} \quad \ce_\nu^+=\ce\cap Q_\nu (L^2).
\]
We can introduce the following operators
\[
S_\nu=\int_{-\infty}^0 |\lm|^{\frac12}d E_\lm \quad \text{and} \quad
T_\nu=\int_{0}^\infty |\lm|^{\frac12} d E_\lm.
\]
We may now introduce a new inner product on $\ce$ by the formula
\[
\inp{\psi}{\va}_\nu=\real\big( (S_\nu+T_\nu)\psi, (S_\nu+T_\nu)\va \big)_2, \quad
\psi,\va\in\ce
\]
and the corresponding norm $\|\cdot\|_\nu$. And we easily see that
\eqref{limit equ} is the Euler-Lagrange equation of the functional
\begin{\equ}\label{limit equ functional}
\Phi_{\nu\ka}(\psi)=\frac12\big(\|Q_\nu\psi\|_\nu^2-\|P_\nu\psi\|_\nu^2\big)
-\frac{\ka}{p}|\psi|_{p}^{p}.
\end{\equ}

\begin{Lem}\label{H5}
If $\{\psi_n\}\subset \ce$ is a bounded sequence such that
\[
\Phi_{\nu\ka}'(\psi_n)\to0 \quad \text{and} \quad
\liminf_{n\to\infty}|\psi_n|_{p}>0.
\]
Then there exists $\psi\neq0$ with $\Phi_{\nu\ka}'(\psi)=0$.
\end{Lem}
\begin{proof}
Let $B_R^0$ denote the open ball of radius $R$ centered at the origin. If
\[
\lim_{n\to\infty}\sup_{y\in \R^m}\int_{y+B_R^0} |\psi_n|^2dx=0, \quad \forall R>0,
\]
then by  Lions' result \cite{Lions} $\psi_n\to0$ in $L^{q}$ for all $q\in(2,m^*)$
and therefore $|\psi_n|_p\to0$, which is a contradiction.

Passing to a subsequence, we have
\[
\liminf_{n\to\infty}\int_{y_n+B_R^0}|\psi_n|^2dx>0
\]
for some $R>0$ and $\{y_n\}\subset\R^m$. Using the invariance of the operator $A_\nu$
under translations, we can find $R>0$ and a new sequence $\{\tilde\psi_n\}$ such that
\[
\Phi_{\nu\ka}'(\tilde\psi_n)\to0 \quad \text{and} \quad
\liminf_{n\to\infty}\int_{B_R^0}|\tilde\psi_n|^2 dx>0.
\]
Up to a subsequence if necessary, we have $\tilde\psi_n\rightharpoonup \psi$
and the compact embedding $\ce\hookrightarrow L^2_{loc}$ shows that
$\psi\neq0$. Note that $|\tilde\psi_n|^{p-2}\tilde\psi_n\rightharpoonup |\psi|^{p-2}\psi$
in $L^{\frac{p}{p-1}}$, by taking the limit in $\Phi_{\nu\ka}'(\tilde\psi_n)\to0$, we obtain
$\Phi_{\nu\ka}'(\psi)=0$ as desired.
\end{proof}

\begin{Cor}
For each $\nu,\ka>0$, there exists a nontrivial solution $\psi\in\ce$ to Eq. \eqref{limit equ}.
\end{Cor}
\begin{proof}
By Lemma \ref{H5}, this is a direct consequence of \cite[Theorem 2.1]{BJS}.
\end{proof}

Now we may define
\[
\ga(\nu,\ka)=\inf\big\{ \Phi_{\nu\ka}(\psi): \, \psi\in\ce\setminus\{0\} \text{ s.t. }
\Phi_{\nu\ka}'(\psi)=0 \big\}.
\]
Because the super-quadratic part in \eqref{limit equ functional} is simply the $L^{p}$ norm,
we easily see that $\ga(\nu,\ka)>0$ is attained. Particularly, by \cite{BJS} and a similar argument as
of Lemma \ref{reduction2}, the following reduction principle holds.
\begin{Lem}\label{reduction limit equ}
For each $\nu,\ka>0$,
\begin{itemize}
\item[$(1)$] there exists a $C^1$ map $h_{\nu\ka}:\ce_\nu^+\to\ce_\nu^-$ such that
$\displaystyle\Phi_{\nu\ka}(u+h_{\nu\ka}(u))=\max_{v\in\ce_\nu^-}\Phi_{\nu\ka}(u+v)$;

\item[$(2)$] denoted by $J_{\nu\ka}(u)=\Phi_{\nu\ka}(u+h_{\nu\ka}(u))$, then critical
points of $J_{\nu\ka}$ and $\Phi_{\nu\ka}$ are in one-to-one correspondence via the injective
map $u\mapsto u+h_{\nu\ka}(u)$;

\item[$(3)$] for each $u\in\ce_\nu^+\setminus\{0\}$, the map $t\mapsto J_{\nu\ka}(tu)$
has only one maximum on $(0,+\infty)$ and
$\displaystyle \ga(\nu,\ka)=\inf_{u\in\ce_\nu^+\setminus\{0\}}
\max_{t>0}J_{\nu\ka}(tu)$.
\end{itemize}
\end{Lem}

In the next step, we will study the behavior of the map $(\nu,\ka)\mapsto\ga(\nu,\ka)$. Specifically,
the monotonicity of $\ga$ with respect to the two parameters is at the core of this paper.

\begin{Prop}\label{euclidean prop}
$\ga(\nu,\ka)=\nu^{-(m-1)+\frac2{p-2}}\,\ka^{-\frac2{p-2}}\,\ga(1,1)$.
%
\end{Prop}
\begin{proof}
%
In fact, taking $\rho>0$ as a parameter, we can assert that:
$\psi$ is a nontrivial solution of Eq. \eqref{limit equ} with energy $\ga(\nu,\ka)$
if and only if $\va(x)=\rho\psi(x/\nu)$ solves
\[
\tilde D_{\ig_{\R^m}}\va + \om_\C\cdot\va = \frac{\nu^{-1}\ka}{\rho^{p-2}} |\va|^{p-2}\va \quad
\text{on } \R^m
\]
with the energy
\[
\ga\Big(1,\frac{\nu^{-1}\ka}{\rho^{p-2}}\Big)=\nu^{m-1}\, \rho^2\, \ga(\nu,\ka).
\]
Therefore, the conclusion follows easily by substituting $\rho=(\nu^{-1}\ka)^{\frac1{p-2}}$.
\end{proof}

\begin{Rem}
Since $p\in(2,m^*)$, we have $-(m-1)+\frac2{p-2}>0$ and the value
of $\ga(\nu,\ka)$ decreases as $\nu$ goes smaller and $\ka$ goes larger.
\end{Rem}


\section{Bubbling analysis}\label{Bubble analysis}

Let us recall that the coefficients $a, b$ in  the model problem Eq. \eqref{equ3} are
positive functions in the class $C^1(M)$. Motivated by Proposition \ref{euclidean prop},
we introduce a potential function $\al:M\to\R$ as
\[
\al=a^{-(m-1)+\frac2{p-2}}\,b^{-\frac2{p-2}},
\]
and we write $\displaystyle\al_{min}=\min_M\al$. The set of minimum
points of $\al$ will be denoted by
\begin{\equ}\label{cc}
\cc=\big\{ \xi\in M:\, \al(\xi)=\al_{min}\big\}.
\end{\equ}

Fix $\xi_0\in\cc$ arbitrarily, we can denote $\nu_0=a(\xi_0)$ and
$\ka_0=b(\xi_0)$. Then, by  Lemma \ref{reduction limit equ}, there exists
$\psi_0\in\ce$ such that
\begin{\equ}\label{test equ}
\tilde D_{\ig_{\R^m}}\psi_0+\nu_0\om_\C\cdot \psi_0=\ka_0 |\psi_0|^{p-2}\psi_0
\quad \text{on } \R^m
\end{\equ}
and
\begin{\equ}\label{test spinor energy}
\Phi_{\nu_0\ka_0}(\psi_0)=\ga(\nu_0,\ka_0)\equiv\al_{min}\ga(1,1).
\end{\equ}
Let $\eta\in C^\infty(\R^m)$ be such that $\eta(x)=1$ for $|x|\leq 1/2$
and $\eta(x)=0$ for $|x|\geq1$. We define a spinor $\psi_\vr\in \Ga(\tilde\mbs(\R^m))$
by
\[
\va_\vr(x)=\eta_\vr(x)\psi_0(x) \quad \text{where} \quad
\eta_\vr(x)=\eta(\vr^{\frac12}x).
\]

Suppose $\xi_0\in V\subset M$ and let $(x_1,\dots, x_m)$ be the normal coordinates
given by the exponential map $\exp_{\xi_0}: U\subset T_{\xi_0}M\cong \R^m\to V$,
$x\mapsto y=\exp_{\xi_0}x$. We define
\[
\mu_\vr(x)=\exp_{\xi_0}(\vr x)
\]
 such that $\vr|x|<inj_M$, where $inj_M>0$ is the injectivity radius of $M$.

Denoted by $B_R^0=\{x\in\R^m:\, |x|<R\}$, where $|\cdot|$ is the Euclidean norm
in $\R^m$, we have a conformal equivalence $(B_{\vr^{-1/2}}^0, \vr^{-2}\mu_\vr^*\ig)
\cong (B_{\vr^{1/2}}(\xi_0),\ig)\subset M$ for all $\vr$ small.

For ease of notation, we set $\ig_\vr=\vr^{-2}\mu_\vr^*\ig$. Writing the metric $\ig$
in geodesic normal coordinates centered at $\xi_0$, one immediately sees that
$\ig_\vr$ converges to the Euclidean metric in $C^\infty$-topology on $B_{\vr^{-1/2}}^0$.

We point out here that, by using the idea of Bourguignon-Gauduchon trivialization \cite{BG}
(see also \cite{AGHM}),
 the coordinate map $\mu_\vr$ induces a bundle identification
$\ov{(\mu_\vr)}_*: \mbs_x(B_{\vr^{-1/2}}^0, \ig_\vr)\to
\mbs_{\mu_\vr(x)}(B_{\vr^{1/2}}(\xi_0),\ig)$. Hence we can define spinors
on $B_{\vr^{1/2}}(\xi_0)$ by
\begin{\equ}\label{test spinor}
\phi_\vr:=\ov{(\mu_\vr)}_*\circ \va_\vr \circ \mu_\vr^{-1}.
\end{\equ}
Then, by the transformation property of the Dirac operator under conformal
change of the metric (see \cite{Hij86, Hit74}), a straightforward calculation shows that
\begin{\equ}\label{X1}
\vr \tilde D_\ig\phi_\vr=\ov{(\mu_\vr)}_*\circ (D_{\ig_\vr}\va_\vr)\circ \mu_\vr^{-1},
\end{\equ}
and moreover,
\begin{\equ}\label{X2}
\frac1{\vr^m}\int_{B_{\vr^{1/2}}(\xi_0)}\vr(\tilde D_\ig\phi_\vr,\phi_\vr)d\vol_\ig
=\int_{B_{\vr^{-1/2}}^0}(D_{\ig_\vr}\va_\vr,\va_\vr)d\vol_{\ig_\vr},
\end{\equ}
\begin{\equ}\label{X3}
\frac1{\vr^m}\int_{B_{\vr^{1/2}}(\xi_0)}(\om_\C\cdot_\ig\phi_\vr,\phi_\vr) d\vol_\ig
=\int_{B_{\vr^{-1/2}}^0}(\om_\C\cdot_{\ig_\vr}\va_\vr,\va_\vr) d\vol_{\ig_\vr},
\end{\equ}
\begin{\equ}\label{X4}
\frac1{\vr^m}\int_{B_{\vr^{1/2}}(\xi_0)}|\phi_\vr|^{p}d\vol_\ig
=\int_{B_{\vr^{-1/2}}^0} |\va_\vr|^{p} d\vol_{\ig_\vr},
\end{\equ}
where $\om_\C\cdot_\ig$ and $\om_\C\cdot_{\ig_\vr}$ denote the Clifford
multiplication by the chirality operators with respect to the metric $\ig$ and $\ig_\vr$
respectively.

\begin{Lem}\label{d-cl-vr to 0}
$\|\cl_\vr'(\phi_\vr)\|_\vr\to0$ as $\vr\to0$.
\end{Lem}
\begin{proof}
Let $\va\in\ch$ be an arbitrary test spinor, it follows that
\begin{\equ}\label{d1}
\cl_\vr(\phi_\vr)[\va]=\frac1{\vr^m}\real \int_M(A\phi_\vr,\va)-b|\phi_\vr|^{p-2}(\phi_\vr,\va)d\vol_\ig.
\end{\equ}
Notice that $\phi_\vr=\ov{(\mu_\vr)}_*\circ\va_\vr\circ\mu_\vr^{-1}$, we have
\[
\vr\tilde D_\ig\phi_\vr=\ov{(\mu_\vr)}_*\circ( \nabla\eta_\vr\cdot_{\ig_\vr}\psi_0 )
\circ \mu_\vr^{-1}+\ov{(\mu_\vr)}_*\circ( \eta_\vr \tilde D_{\ig_\vr}\psi_0) \circ
\mu_\vr^{-1}
\]
where $\cdot_{\ig_\vr}$ is the Clifford multiplication with respect to the metric $\ig_\vr$.
Substituting this into \eqref{d1}, we get
\begin{\equ}\label{d2}
\cl_\vr'(\phi_\vr)[\va]=l_1+l_2+l_3+l_4
\end{\equ}
where
\[
\aligned
l_1&=\frac1{\vr^m}\real\int_M\big( \ov{(\mu_\vr)}_*\circ( \nabla\eta_\vr\cdot_{\ig_\vr}\psi_0)
\circ \mu_\vr^{-1},\va ) d\vol_\ig \\[0.5em]
&=\real\int_{B^0_{\vr^{-1/2}}}\big( \nabla\eta_\vr\cdot_{\ig_\vr}\psi_0,
\ov{(\mu_\vr)}_*^{-1}\circ\va\circ\mu_\vr \big) d\vol_{\ig_\vr},
\endaligned
\]
\[
\aligned
l_2&=\frac1{\vr^m}\real\int_M(\eta_\vr\circ\mu_\vr^{-1})\big( \ov{(\mu_\vr)}_*\circ(
\tilde D_{\ig_\vr}\psi_0- \tilde D_{\ig_{\R^m}}\psi_0 )\circ\mu_\vr^{-1},\va) d\vol_\ig \\[0.5em]
&=\real\int_{B^0_{\vr^{-1/2}}}\eta_\vr\cdot\big(\tilde D_{\ig_\vr}\psi_0- \tilde D_{\ig_{\R^m}}\psi_0,
\ov{(\mu_\vr)}_*^{-1}\circ\va\circ\mu_\vr \big) d\vol_{\ig_\vr},
\endaligned
\]
\[
\aligned
l_3&=\frac1{\vr^m}\real\int_M(\eta_\vr\circ\mu_\vr^{-1})\big( \ov{(\mu_\vr)}_*\circ(
\tilde D_{\ig_{\R^m}}\psi_0+ a\om_\C\cdot_{\ig_\vr}\psi_0-b|\psi_0|^{p-2}\psi_0
)\circ \mu_\vr^{-1}, \va \big) d\vol_\ig  \\[0.5em]
&=\real\int_{B^0_{\vr^{-1/2}}}\eta_\vr\cdot\big( \tilde D_{\ig_{\R^m}}\psi_0+ (a\circ\mu_\vr)\om_\C\cdot_{\ig_\vr}
\psi_0-(b\circ\mu_\vr)|\psi_0|^{p-2}\psi_0,
\ov{(\mu_\vr)}_*^{-1}\circ\va\circ\mu_\vr \big) d\vol_{\ig_\vr},
\endaligned
\]
and
\[
\aligned
l_4&=\frac1{\vr^m}\real\int_M\big( \ov{(\mu_\vr)}_*\circ(b\cdot\eta_\vr|\psi_0|^{p-2}\psi_0
-b\cdot\eta_\vr^{p-1}|\psi_0|^{p-2}\psi_0)\circ\mu_\vr^{-1},\va ) d\vol_\ig \\[0.5em]
&=\real\int_{B^0_{\vr^{-1/2}}}(b\circ\mu_\vr)(\eta_\vr-\eta_\vr^{p-1})|\psi_0|^{p-2}
(\psi_0,\ov{(\mu_\vr)}_*^{-1}\circ\va\circ\mu_\vr ) d\vol_{\ig_\vr}.
\endaligned
\]

For $l_1$, by the H\"older inequality and Lemma \ref{embeding lemma}, we have
\begin{eqnarray*}
|l_1|&\leq& \Big( \int_{B^0_{\vr^{-1/2}}}|\nabla\eta_\vr\cdot_{\ig_\vr}\psi_0|^2 d\vol_{\ig_\vr}
\Big)^{\frac12} \Big( \frac1{\vr^m}\int_{B^0_{\vr^{-1/2}}}
|\va\circ\mu_\vr|^2 d\vol_{\mu_\vr^*\ig}\Big)^\frac12     \\
&\leq& C\vr^{\frac12} \Big( \int_{B^0_{\vr^{-1/2}}\setminus B^0_{\frac12\vr^{-1/2}}}|\psi_0|^2 d\vol_{\ig_{\R^m}}
\Big)^{\frac12} \cdot |\va|_{2,\vr}   \\
&\leq& C\vr^{\frac12} |\psi_0|_2\cdot \|\va\|_\vr,
\end{eqnarray*}
where we used $d\vol_{\ig_\vr}\leq C d\vol_{\ig_{\R^m}}$ on $B^0_{\vr^{-1/2}}$ for some
constant $C>0$ as $\vr\to0$. We soon obtain
\begin{\equ}\label{l1}
|l_1|\leq o_\vr(1)\|\va\|_\vr \quad \text{as } \vr\to0.
\end{\equ}

In order to estimate $l_2$, let us mention that, by \eqref{test equ} and the $L^p$-theory for
Dirac operators, we have $\nabla\psi_0\in L^{\frac{p}{p-1}}(\R^m,\tilde\mbs(\R^m))$.
And hence, we get
\begin{eqnarray*}
|l_2|&\leq& \Big(\int_{B^0_{\vr^{-1/2}}}|\tilde D_{\ig_\vr}\psi_0-\tilde D_{\ig_{\R^m}}\psi_0|^\frac{p}{p-1}
d \vol_{\ig_\vr} \Big)^{\frac{p-1}{p}} \Big( \frac1{\vr^m}\int_{B^0_{\vr^{-1/2}}} |\va\circ\mu_\vr|^{p}
d\vol_{\mu_\vr^*\ig}\Big)^{\frac1{p}}   \\
&\leq& C \Big(\int_{B^0_{\vr^{-1/2}}}|\tilde D_{\ig_\vr}\psi_0-\tilde D_{\ig_{\R^m}}\psi_0|^\frac{p}{p-1}
d \vol_{\ig_{\R^m}} \Big)^{\frac{p-1}{p}} \|\va\|_\vr .
\end{eqnarray*}
Since $\ig_\vr\to\ig_{\R^m}$ on $B^0_{\vr^{-1/2}}$ in $C^\infty$-topology as $\vr\to0$, we can get
further from the above estimate that
\begin{\equ}\label{l2}
|l_2|\leq o_\vr(1)\|\va\|_\vr \quad \text{as } \vr\to0.
\end{\equ}

The estimate for $l_3$ is much more clear. Indeed, by the definition of $\mu_\vr$, we have
\[
a\circ\mu_\vr \to\nu_0 \quad \text{and} \quad b\circ\mu_\vr\to\ka_0
\]
uniformly on $B^0_{\vr^{-1/2}}$ and, therefore, it follows that
\begin{eqnarray}\label{l3}
|l_3|&\leq& C\Big(\int_{B^0_{\vr^{-1/2}}} |a\circ\mu_\vr-\nu_0|^2|\psi_0|^2d\vol_{\ig_{\R^m}}
\Big)^{\frac12}|\va|_{2,\vr}    \nonumber \\
& & +C\Big( \int_{B^0_{\vr^{-1/2}}} |b\circ\mu_\vr-\ka_0|^{\frac{p}{p-1}}|\psi_0|^{p}d\vol_{\ig_{\R^m}}
\Big)^{\frac{p-1}{p}}|\va|_{p,\vr}  \nonumber \\
&\leq& o_\vr(1)\|\va\|_\vr
\end{eqnarray}
as $\vr\to0$.

It remains to estimate $l_4$. Similarly as was argued in the above, we have
\begin{eqnarray*}
|l_4|&\leq& C \Big(\int_{B^0_{\vr^{-1/2}}} (\eta_\vr-\eta_\vr^{p-1})^{\frac{p}{p-1}}|\psi_0|^{p}
d\vol_{\ig_{\R^m}} \Big)^{\frac{p-1}{p}} |\va|_{p,\vr}  \\
&\leq& C\Big(\int_{B^0_{\vr^{-1/2}}\setminus B^0_{\frac12\vr^{-1/2}}} |\psi_0|^{p}
d\vol_{\ig_{\R^m}} \Big)^{\frac{p-1}{p}} \|\va\|_\vr
\end{eqnarray*}
where
\[
\int_{B^0_{\vr^{-1/2}}\setminus B^0_{\frac12\vr^{-1/2}}} |\psi_0|^{p}
d\vol_{\ig_{\R^m}}\to0
\]
as $\vr\to0$. Hence we have
\begin{\equ}\label{l4}
|l_4|\leq o_\vr(1)\|\va\|_\vr \quad \text{as } \vr\to0.
\end{\equ}

Combining \eqref{l1}-\eqref{l4}, we have $\|\cl_\vr'(\phi_\vr)\|_\vr\to0$ as $\vr\to0$
as desired.
\end{proof}

\begin{Lem}\label{cl to ga}
$\cl_\vr(\phi_\vr)\to\ga(\nu_0,\ka_0)$ as $\vr\to0$.
\end{Lem}
\begin{proof}
Due to the definition of the test spinor $\phi_\vr$ in \eqref{test spinor}, we can find
a constant $C>0$ independent of $\vr$ such that $\|\phi_\vr\|_\vr\leq C$.
And thus, by Lemma \ref{d-cl-vr to 0}, we easily see that
\begin{\equ}
\cl_\vr(\phi_\vr)=\cl_\vr(\phi_\vr)-\frac12\cl_\vr'(\phi_\vr)[\phi_\vr]+o_\vr(1)=
\frac{p-2}{2\vr^m\, p}\int_M b |\phi_\vr|^{p}d\vol_\ig+o_\vr(1)
\end{\equ}
as $\vr\to0$.

Observe that
\[
\frac1{\vr^m}\int_M b |\phi_\vr|^{p}d\vol_\ig=\frac1{\vr^m}\int_{B_{\vr^{1/2}}(\xi_0)}
b |\phi_\vr|^{p}d\vol_\ig
=\int_{B^0_{\vr^{-1/2}}}b\circ\mu_\vr|\eta_\vr\cdot\psi_0|^{p}d\vol_{\ig_\vr},
\]
where
\[
\int_{B^0_{\vr^{-1/2}}}b\circ\mu_\vr|\eta_\vr\cdot\psi_0|^{p}d\vol_{\ig_\vr}
=\ka_0\int_{\R^m}|\psi_0|^{p}d\vol_{\ig_{\R^m}}+o_\vr(1)
\]
as $\vr\to0$.

Thus, by \eqref{test equ} and \eqref{test spinor energy}, we soon have
\[
\cl_\vr(\phi_\vr)=\frac{p-2}{2p}\ka_0\int_{\R^m}|\psi_0|^{p}d\vol_{\ig_{\R^m}}+o_\vr(1)
=\ga(\nu_0,\ka_0)+o_\vr(1)
\]
as $\vr\to0$ which completes the proof.
\end{proof}

Here, we emphasize that the above two lemmas yield a description of the limiting behavior
of the critical value $\ga_\vr$ obtained in \eqref{ga-vr}. Namely, applying Corollary \ref{key corollary}
for $\{\phi_\vr\}$, we obtain the following estimate.

\begin{Cor}\label{ga-vr leq}
$\displaystyle\limsup_{\vr\to0}\ga_\vr\leq \ga(\nu_0,\ka_0)=\al_{min}\ga(1,1)$.
\end{Cor}

\section{Asymptotic profiles of the solutions}

In the present section we will firstly establish a complete description for
the model problem \eqref{equ3}.
After that we will return to the analysis of the limiting behavior of the conformal
metrics induced by \eqref{equ2}.

To begin with, let $\{\psi_\vr\}$ be a family of solutions to \eqref{equ3}
found by \eqref{ga-vr}, i.e.,
\begin{\equ}\label{xx0}
\cl_\vr(\psi_\vr)=\ga_\vr \quad \text{and} \quad \cl_\vr'(\psi_\vr)=0.
\end{\equ}
Then, as was mentioned in \eqref{ga-vr geq tau}, we find
\begin{\equ}\label{xx1}
\frac{p-2}{2\vr^m\, p}\int_M b |\psi_\vr|^{p}d\vol_\ig=\ga_\vr\geq\tau_0>0
\end{\equ}
for some $\tau_0>0$.
In what follows, for any $\xi\in M$ and $r>0$, $B_r(\xi)\subset M$ denote the distance
ball of radius $r$ with respect to the metric $\ig$.

\begin{Lem}\label{step1}
There exist $\xi_\vr\in M$, $r_0,\de_0>0$ such that
\[
\liminf_{\vr\to0}\frac1{\vr^m}\int_{B_{\vr r_0}(\xi_\vr)}|\psi_\vr|^2d\vol_\ig \geq \de_0.
\]
\end{Lem}
\begin{proof}
Assume on the contrary that for any $r>0$
\begin{\equ}\label{a2}
\sup_{\xi\in M}\frac1{\vr^m}\int_{B_{2\vr r}(\xi)}|\psi_\vr|^2 d\vol_\ig \to0
\quad \text{as } \vr\to0.
\end{\equ}
For each $\xi\in M$, let us now choose a smooth real cut-off function
$\chi_{\xi,\vr}\equiv1$ on $B_{\vr r}(\xi)$ and $\supp\chi_{\xi,\vr}\subset B_{2\vr r}(\xi)$.
Then, for $s\in(0,1)$, we consider $q_s=2+(m^*-2)s\in(2,m^*)$ and we have
\[
\int_{B_{2\vr r}(\xi)}|\chi_{\xi,\vr}\psi_\vr|^{q_s}d\vol_\ig\leq
\Big( \int_{B_{2\vr r}(\xi)}|\chi_{\xi,\vr}\psi_\vr|^2 d\vol_\ig\Big)^{1-s}
\Big( \int_{B_{2\vr r}(\xi)}|\chi_{\xi,\vr}\psi_\vr|^{\frac{2m}{m-1}} d\vol_\ig\Big)^s.
\]
Taking $s=\frac2{m^*}$, we obtain from Lemma \ref{embeding lemma} that
\[
\Big( \frac1{\vr^m}\int_{B_{2\vr r}(\xi)}|\chi_{\xi,\vr}\psi_\vr|^{m^*} d\vol_\ig\Big)^s
\leq C \|\chi_{\xi,\vr}\psi_\vr\|_\vr^2.
\]
Now, covering $M$ by balls of radius $\vr r$ such that any point $\xi\in M$ is contained
in at most $K_M$ balls, where $K_M$ does not depend on $\vr$. This condition can be
satisfied for $\vr$ small by the compactness of $M$. And thus, we find
\[
\frac1{\vr^m}\int_M|\psi_\vr|^{q_s}d\vol_\ig\leq C\cdot K_M\Big(
\sup_{\xi\in M}\int_{B_{2\vr r}(\xi)}|\chi_{\xi,\vr}\psi_\vr|^2 d\vol_\ig\Big)^{1-s}
\|\psi_\vr\|_\vr^2.
\]
Notice that $\|\psi_\vr\|_\vr$ is bounded, it follows from \eqref{a2} that
$|\psi_\vr|_{q_s,\vr}\to0$. Since $2<q_s<m^*$, we see easily that
$|\psi_\vr|_{q,\vr}\to0$ for all $q\in(2,m^*)$ which contradict to \eqref{xx1}
\end{proof}

We may now fix the sequence $\{\xi_\vr\}\subset M$ and the constants $r_0, \de_0>0$
in Lemma \ref{step1}. Up to a subsequence if necessary, we assume that $\xi_\vr\to\xi_\infty\in M$
 as $\vr\to0$. Then, similar to Section \ref{Bubble analysis},
we can consider the rescaled geodesic normal
coordinates near each $\xi_\vr$ via the formula
\[
\Theta_\vr(x)=\exp_{\xi_\vr}(\vr x).
\]
For any $R\geq r_0$, we have a conformal equivalence $(B_R^0, \vr^{-2}\Theta_\vr^*\ig)
\cong (B_{\vr R}(\xi_\vr),\ig)$ for all small $\vr$.
Denoting by $\ig_{\Theta_\vr}=\vr^{-2}\Theta_\vr^*\ig$, then we have
$\ig_{\Theta_\vr}\to\ig_{\R^m}$ in $C^\infty(B_R^0)$ as $\vr\to0$.

Let $\ov{(\Theta_\vr)}_*:\mbs_x(B_R^0,\ig_{\Theta_\vr})\to\mbs_{\Theta_\vr(x)}
(B_{\vr R}(\xi_\vr),\ig)$ denote the bundle identification for spinors, we can introduce
a family of spinors on $B_R^0$ by
\begin{\equ}\label{spinors z}
z_\vr=\ov{(\Theta_\vr)}_*^{-1}\circ\psi_\vr\circ\Theta_\vr.
\end{\equ}
Along this line of consideration, we have
\[
\tilde D_{\ig_{\Theta_\vr}}z_\vr=\ov{(\Theta_\vr)}_*^{-1}\circ(\vr\tilde D_\ig \psi_\vr)
\circ \Theta_\vr,
\]
\begin{\equ}\label{z2}
\int_{B_R^0}(\tilde D_{\ig_{\Theta_\vr}}z_\vr,z_\vr)d\vol_{\ig_{\Theta_\vr}}=
\frac1{\vr^m}\int_{B_{\vr R}(\xi_\vr)}(\vr \tilde D_\ig \psi_\vr,\psi_\vr) d\vol_\ig,
\end{\equ}
\begin{\equ}\label{z3}
\int_{B_R^0}(\om_\C\cdot_{\ig_{\Theta_\vr}}z_\vr,z_\vr)d\vol_{\ig_{\Theta_\vr}}=
\frac1{\vr^m}\int_{B_{\vr R}(\xi_\vr)}(\om_\C\cdot_{\ig}\psi_\vr,\psi_\vr) d\vol_\ig,
\end{\equ}
\begin{\equ}\label{z4}
\int_{B_R^0}|z_\vr|^{p}d\vol_{\ig_{\Theta_\vr}}=
\frac1{\vr^m}\int_{B_{\vr R}(\xi_\vr)}|\psi_\vr|^{p} d\vol_\ig.
\end{\equ}
Particularly, since $\|\psi_\vr\|_\vr$ is bounded, we have
\begin{\equ}\label{z5}
0<\liminf_{\vr\to0}\int_{B_R^0}|z_\vr|^{p}d\vol_{\ig_{\Theta_\vr}}\leq
\limsup_{\vr\to0}\frac1{\vr^m}\int_M|\psi_\vr|^{p}d\vol_\ig<\infty
\end{\equ}
for any $R>r_0$.

Recalling the set $\cc\subset M$ defined in \eqref{cc}, we have
\begin{Lem}\label{step2}
$\xi_\infty\in \cc$, i.e., $\dist_{\ig}(\xi_\vr,\cc)\to0$ as $\vr\to0$.
\end{Lem}
\begin{proof}
Since $\{z_\vr\}$ is $W^{\frac12,2}_{loc}(\R^m,\tilde\mbs(\R^m))$-bounded, that is,
$\{\bt z_\vr\}\subset W^{\frac12,2}(\R^m,\tilde\mbs(\R^m))$ is bounded
for any $\bt\in C_c^\infty(\R^m)$. We can assume, up to a subsequence,
$z_\vr\rightharpoonup z_\infty$ in $W^{\frac12,2}_{loc}(\R^m,\tilde\mbs(\R^m))$
and $z_\vr\to z_\infty$ in $L_{loc}^q(\R^m,\tilde\mbs(\R^m))$ for $2\leq q<\frac{2m}{m-1}$.
At the same time, by \eqref{z5}, we see easily that $z_\infty\in L^{p}(\R^m,\tilde\mbs(\R^m))$.

Let $\va\in W^{\frac12,2}(\R^m,\tilde\mbs(\R^m))$ be such that
$\supp \va$ is compact, i.e. $\supp \va\subset B_R^0$ for some $R$ large.
Then, we have
\[
\aligned
&\int_{\R^m}\big( \tilde D_{\ig_{\R^m}}z_\infty + a(\xi_\infty)\om_\C\cdot_{\ig_{\R^m}}z_\infty
-b(\xi_\infty)|z_\infty|^{p-2}z_\infty, \va \big)d\vol_{\ig_{\R^m}}  \\
&\qquad =\lim_{\vr\to0}\int_{\supp\va}\big( \tilde D_{\ig_{\Theta_\vr}}z_\vr
+(a\circ\Theta_\vr)\om_\C\cdot_{\ig_{\Theta_\vr}}z_\vr-(b\circ\Theta_\vr) |z_\vr|^{p-2}z_\vr, \va\big)
d\vol_{\ig_{\Theta_\vr}}  \\
&\qquad =\lim_{\vr\to0}\frac1{\vr^m}\int_{B_{\vr R}(\xi_\vr)}
\big( \vr \tilde D_\ig \psi_\vr+ a\om_\C\cdot_\ig \psi_\vr
-b|\psi_\vr|^{p-2}\psi_\vr, \ov{(\Theta_\vr)}_*\circ\va\circ\Theta_\vr^{-1} \big) d\vol_\ig\\
&\qquad =0
\endaligned
\]
Hence, we have $z_\infty$ satisfies
\begin{\equ}\label{z-infty}
\tilde D_{\ig_{\R^m}}z_\infty + a(\xi_\infty)\om_\C\cdot_{\ig_{\R^m}}z_\infty
=b(\xi_\infty)|z_\infty|^{p-2}z_\infty \quad \text{on } \R^m.
\end{\equ}
This implies, by the elliptic regularity, $\tilde D_{\ig_{\R^m}}z_\infty + a(\xi_\infty)\om_\C\cdot_{\ig_{\R^m}}z_\infty\in L^{\frac{p}{p-1}}(\R^m,\tilde\mbs(\R^m))$.
Moreover, combined with the Sobolev embedding
$L^{\frac{p}{p-1}}(\R^m,\tilde\mbs(\R^m))\hookrightarrow
W^{-\frac12,2}(\R^m,\tilde\mbs(\R^m))$, we get $z_\infty\in
W^{\frac12,2}(\R^m,\tilde\mbs(\R^m))$.

Now, by collecting \eqref{xx0}, \eqref{xx1} and \eqref{z2}-\eqref{z5}, we conclude that
$z_\infty$ is a nontrivial solution to \eqref{z-infty} and
\[
\aligned
\limsup_{\vr\to0}\cl_\vr(\psi_\vr)&=\limsup_{\vr\to0}
\frac{p-2}{2p\,\vr^m}\int_M b|\psi_\vr|^{p}d\vol_\ig  \\
&\geq\liminf_{\vr\to0}\frac{p-2}{2p}b(\xi_\infty)\int_{B_R^0}|z_\vr|^{p}d\vol_{\ig_{\Theta_\vr}}\\
&\geq\frac{p-2}{2p}b(\xi_\infty)\int_{B_R^0}|z_\infty|^{p}d\vol_{\ig_{\R^m}}
\endaligned
\]
where in the last inequality we have used the Fatou's lemma.
Due to the arbitrariness of $R>0$, combined with the results obtained in Section \ref{energy gap},
we have
\[
\limsup_{\vr\to0}\cl_\vr(\psi_\vr)\geq \ga\big(a(\xi_\infty),b(\xi_\infty)\big)=\al(\xi_\infty)\ga(1,1).
\]
Then Corollary \ref{ga-vr leq} suggests $\al(\xi_\infty)=\al_{min}$, i.e., $\xi_\infty\in\cc$, which
completes the proof.
\end{proof}

With Lemma \ref{step1} and \ref{step2} in hand, we may now choose $\bt\in C^\infty(M)$
be a cut-off function such that $\bt\equiv1$ on $B_r(\xi_\infty)$ and $\supp\bt\subset B_{2r}(\xi_\infty)$
for some $r>0$ and define spinors on $M$ as
\[
\zeta_\vr=\bt(\cdot)\ov{(\Theta_\vr)}_*\circ z_\infty \circ \Theta_\vr^{-1}.
\]
Setting $w_\vr=\psi_\vr-\zeta_\vr$, we soon have $|w_\vr|_{p,\vr}\to0$ as
$\vr\to0$ (otherwise, we can apply Lemma \ref{step1} and \ref{step2} for $\{w_\vr\}$ instead
of $\{\psi_\vr\}$ to get $\cl_\vr(\psi_\vr)\geq 2\al_{min}\ga(1,1)$ which is absurd).

\begin{Lem}\label{step3}
$\|\cl_\vr'(\zeta_\vr)\|_\vr\to0$ and $\|\cl_\vr'(w_\vr)\|_\vr\to0$ as $\vr\to0$.
\end{Lem}
\begin{proof}
We point out that, after some minor revision, the proof of Lemma \ref{d-cl-vr to 0}
can be applied here to show $\|\cl_\vr'(\zeta_\vr)\|_\vr\to0$ as $\vr\to0$. Hence,
we only need to check the second estimate.

Again, we choose $\va\in\ch$ be an arbitrary test spinor. We then have
\begin{eqnarray}\label{ee1}
\cl_\vr'(w_\vr)[\va]&=&\frac1{\vr^m}\real\int_M(\vr\tilde D_\ig w_\vr+a\om_\C\cdot w_\vr
-b|w_\vr|^{p-2}w_\vr,\va)d\vol_\ig  \nonumber\\
&=&\cl_\vr'(\psi_\vr)[\va]-\cl_\vr'(\zeta_\vr)[\va]+\frac1{\vr^m}\real\int_M(\Psi_\vr,\va)d\vol_\ig,
\end{eqnarray}
where
\[
\Psi_\vr=b|\psi_\vr|^{p-2}\psi_\vr-b|\zeta_\vr|^{p-2}\zeta_\vr-b|w_\vr|^{p-2}w_\vr.
\]
Since $\cl_\vr'(\psi_\vr)=0$ for all $\vr$ small, it remains to estimate the last integral
in \eqref{ee1}.

To proceed, we first mention that there exists $C>0$ (independent of $\vr$) such that
\begin{\equ}\label{ee2}
|\Psi_\vr|\leq C|\zeta_\vr|^{p-2}|w_\vr|+C|w_\vr|^{p-2}|\zeta_\vr|.
\end{\equ}
Thus, for any $R>0$, we have
\[
\aligned
&\frac1{\vr^m}\int_{M\setminus B_{\vr R}(\xi_\vr)}|\zeta_\vr|^{p-2}\cdot|w_\vr|\cdot|\va|d\vol_\ig \\
&\qquad \leq \Big( \frac1{\vr^m}\int_{M\setminus B_{\vr R}(\xi_\vr)}|\zeta_\vr|^{p} d\vol_\ig
\Big)^{\frac{p-2}{p^*}}
\Big( \frac1{\vr^m}\int_{M\setminus B_{\vr R}(\xi_\vr)}|w_\vr|^{p}d\vol_\ig \Big)^{\frac1{p}}
|\va|_{p,\vr}  \\
&\qquad \leq C \Big( \int_{B_{2r/\vr}^0\setminus B_R^0}|z_\infty|^{p}d\vol_{\ig_{\Theta_\vr}}
\Big)^{\frac{p-2}{p}} \|w_\vr\|_\vr\cdot\|\va\|_\vr = o_R(1)\|\va\|_\vr
\endaligned
\]
and
\[
\aligned
&\frac1{\vr^m}\int_{M\setminus B_{\vr R}(\xi_\vr)}|w_\vr|^{p-2}\cdot|\zeta_\vr|\cdot|\va|d\vol_\ig \\
&\qquad \leq \Big( \frac1{\vr^m}\int_{M\setminus B_{\vr R}(\xi_\vr)}|w_\vr|^{p} d\vol_\ig
\Big)^{\frac{p-2}{p}}
\Big( \frac1{\vr^m}\int_{M\setminus B_{\vr R}(\xi_\vr)}|\zeta_\vr|^{p}d\vol_\ig \Big)^{\frac1{p}}
|\va|_{p,\vr}  \\
&\qquad \leq C \Big( \int_{B_{2r/\vr}^0\setminus B_R^0}|z_\infty|^{p}d\vol_{\ig_{\Theta_\vr}}
\Big)^{\frac1{p}} \|w_\vr\|_\vr^{p-2}\cdot\|\va\|_\vr = o_R(1)\|\va\|_\vr,
\endaligned
\]
where $o_R(1)\to0$ as $R\to\infty$.
At the same time, inside $B_{\vr R}(\xi_\vr)$, we have
\[
\aligned
&\frac1{\vr^m}\int_{B_{\vr R}(\xi_\vr)}|\zeta_\vr|^{p-2}\cdot|w_\vr|\cdot|\va|d\vol_\ig \\
&\qquad \leq \Big( \frac1{\vr^m}\int_{B_{\vr R}(\xi_\vr)}|\zeta_\vr|^{p}d\vol_\ig \Big)^{\frac{p-2}{p}}
\Big( \frac1{\vr^m}\int_{B_{\vr R}(\xi_\vr)}|w_\vr|^{p}d\vol_\ig \Big)^{\frac1{p}} |\va|_{p,\vr}\\
&\qquad \leq C\Big( \int_{\R^m}|z_\infty|^{p} d\vol_{\ig_{\R^m}}\Big)^{\frac{p-2}{p}}
\Big( \int_{B_R^0}|z_\vr-z_\infty|^{p}d\vol_{\ig_{\R^m}}\Big)^{\frac1{p}}\cdot\|\va\|_\vr
=o_\vr(1)\|\va\|_\vr
\endaligned
\]
and
\[
\aligned
&\frac1{\vr^m}\int_{B_{\vr R}(\xi_\vr)}|w_\vr|^{p-2}\cdot|\zeta_\vr|\cdot|\va|d\vol_\ig \\
&\qquad \leq \Big( \frac1{\vr^m}\int_{B_{\vr R}(\xi_\vr)}|w_\vr|^{p}d\vol_\ig \Big)^{\frac{p-2}{p}}
\Big( \frac1{\vr^m}\int_{B_{\vr R}(\xi_\vr)}|\zeta_\vr|^{p}d\vol_\ig \Big)^{\frac1{p}} |\va|_{p,\vr}\\
&\qquad \leq C\Big( \int_{B_R^0}|z_\vr-z_\infty|^{p} d\vol_{\ig_{\R^m}}\Big)^{\frac{p-2}{p}}
\Big( \int_{\R^m}|z_\infty|^{p}d\vol_{\ig_{\R^m}}\Big)^{\frac1{p}}\cdot\|\va\|_\vr
=o_\vr(1)\|\va\|_\vr
\endaligned
\]
as $\vr\to0$, where we have used  $z_\vr\rightharpoonup z_\infty$ in
$W^{\frac12,2}_{loc}(\R^m,\tilde\mbs(\R^m))$ and the compact Sobolev embedding
$W^{\frac12,2}_{loc}(\R^m,\tilde\mbs(\R^m))\hookrightarrow L^{p}_{loc}(\R^m,\tilde\mbs(\R^m))$.

Therefore, we can conclude that
\[
\frac1{\vr^m}\real\int_M(\Psi_\vr,\va)d\vol_\ig=o_\vr(1)\|\va\|_\vr
\quad \text{as } \vr\to0.
\]
And hence, by \eqref{ee1}, we have $\|\cl_\vr'(w_\vr)\|_\vr\to0$ as $\vr\to0$.
\end{proof}

At this point, we have the following result which summarizes the concentration
phenomenon of the family $\{\psi_\vr\}$ for the model problem \eqref{equ3}.

\begin{Prop}\label{concentration}
Let $\{\psi_\vr\}$ be the family of solutions to \eqref{equ3} found by \eqref{ga-vr}.
Then there exist a convergent sequence $\{\xi_\vr\}\subset M$, $\xi_\vr\to\xi_\infty$
as $\vr\to0$ and a non-trivial solution $z_\infty$ of Eq. \eqref{z-infty} such that
\[
\al(\xi_\infty)=\al_{min}
\]
and
\begin{\equ}\label{xx}
\psi_\vr=\bt(\cdot)\ov{(\Theta_\vr)}_*\circ z_\infty \circ \Theta_\vr^{-1} + w_\vr
\quad \text{in } \ch
\end{\equ}
where $\|w_\vr\|_\vr\to0$ as $\vr\to0$, $\Theta_\vr(x)=\exp_{\xi_\vr}(\vr x)$
and $\bt\in C^\infty(M)$ is a cut-off function such that $\bt\equiv1$ on $B_r(\xi_\infty)$
and $\supp\bt\subset B_{2r}(\xi_\infty)$, some $r>0$. Moreover, there holds
\[
\lim_{\vr\to0}\ga_\vr=\lim_{\vr\to0}\cl_\vr(\psi_\vr)=\al_{min}\ga(1,1).
\]
\end{Prop}
\begin{proof}
By collecting Lemmas \ref{step1}-\ref{step2}, it remains to show that $\|w_\vr\|_\vr\to0$
as $\vr\to0$ in \eqref{xx}.

Since $\|\cl_\vr'(w_\vr)\|_\vr\to0$ is
already suggested by Lemma \ref{step3}, we soon have
\[
\aligned
\|w_\vr\|_\vr&\leq \frac1{\vr^m}\int_M b|w_\vr|^{p-1}|w_\vr^+-w_\vr^-|d\vol_\ig + o_\vr(1) \\
&\leq \Big(\frac1{\vr^m}\int_M b |w_\vr|^{p}d\vol_\ig \Big)^{\frac{p-1}{p}}
\Big(\frac1{\vr^m}\int_M b |w_\vr^+-w_\vr^-|^{p}d\vol_\ig \Big)^{\frac1{p}}+o_\vr(1) \\
&\leq C |w_\vr|_{p,\vr} \|w_\vr\|+o_\vr(1)
\endaligned
\]
as $\vr\to0$. As was remarked before Lemma \ref{step3}, we have $|w_\vr|_{p,\vr}\to0$.
Thus, we can infer $\|w_\vr\|_\vr\to0$ as $\vr\to0$.
\end{proof}

The above proposition yields a description of the profiles of the solutions to our
original problem \eqref{equ2}. Namely, we can simply substitute $m=m_1$, $p=n^*=\frac{2(m_1+m_2)}{m_1+m_2-1}$
\[
a=\lm\theta \quad \text{and} \quad b=\theta^{m_1-\frac{m_1-1}2n^*}
\]
into Eq. \eqref{equ3} and calculate the potential function as
\[
\al=a^{-(m_1-1)+\frac2{n^*-2}}\,b^{-\frac2{n^*-2}}\equiv \lm^{m_2} \quad \text{on } M_1.
\]
Applying the same argument that we have done
previously, we obtain our main result as a corollary of Proposition \ref{concentration}.
(cf. \cite[Chapter 3]{Ammann} for regularity results of Dirac operators).

\begin{Thm}\label{result for n geq 3}
There exists $\vr_0>0$ such that, for any $\vr\in(0,\vr_0)$, Eq. \eqref{equ2}
has a solution $\psi_\vr\in C^1(M_1,\tilde\mbs(M_1))\cap C^\infty(M_1\setminus\psi_\vr^{-1}(0), \tilde\mbs(M_1))$. Furthermore, there exist a convergent sequence
$\{\xi_\vr\}\subset M_1$, $\xi_\vr\to\xi_0$
as $\vr\to0$ and a non-trivial solution $z_0$ of
\[
\tilde D_{\ig_{\mbox{\tiny $\R^{m_1}$}}}z+ \lm\theta(\xi_0)\om_\C\cdot_{\ig_{\mbox{\tiny $\R^{m_1}$}}}
z=\theta(\xi_0)^{m_1-\frac{m_1-1}2n^*}|z|^{n^*-2}z
\quad \text{on } \R^{m_1}
\]
such that
\[
\psi_\vr=\bt(\cdot)\ov{(\Theta_\vr)}_*\circ z_0 \circ \Theta_\vr^{-1} + w_\vr
\quad \text{in } \ch
\]
where $\|w_\vr\|_\vr\to0$ as $\vr\to0$, $\Theta_\vr(x)=\exp_{\xi_\vr}(\vr x)$
and $\bt\in C^\infty(M_1)$ is a cut-off function such that $\bt\equiv1$ on $B_r(\xi_\infty)$
and $\supp\bt\subset B_{2r}(\xi_\infty)$, some $r>0$. Moreover, there holds
\[
\lim_{\vr\to0}\frac1{\vr^{m_1}}\int_{M_1} \theta^{m_1-\frac{m_1-1}2n^*}|\psi_\vr|^{n^*}
d\vol_\ig=2n\cdot\lm^{m_2}\cdot\ga(1,1).
\]
\end{Thm}



\section{Application: CMC immersions for Unduloids}\label{S1-bundles}

In this section, comparing with that we have considered previously,
we will consider the simplest case of the product construction.


%

For a given positive $\ell$, we can choose a parameter $t$ on $S^1$ via the
identification $S^1=\R/2\pi\ell\Z$. And from now on, $\ell S^1$ will stand
for this parametrization. Then spinors on $\ell S^1$ can be viewed as complex
vector functions on $\R$ which are periodic and such that $2\pi\ell$ is a period.
Particularly, on $N=\ell S^1\times S^1$, we may write $\ig=dt^2\op d\tau^2$,
where $d\tau^2$ is the standard metric on the second factor with total length $2\pi$.

As was shown in Section \ref{algebraic settings}, one dimension spinor space
is simply $\C$ and the Clifford multiplication by positively oriented unit vector
is multiplication by $i$. With respect to the coordinates $t$ of $\ell S^1$ and $\tau$
of $S^1$, the Dirac operator on $M$  can be written as
\[
D_{\ig}\psi=i\Big( \frac{d}{dt}\psi_1\op -\frac{d}{dt}\psi_2\Big)\otimes\va
-(\psi_2\op -\psi_1)\otimes \frac{d}{d\tau}\va
\]
for all $\psi=(\psi_1\op\psi_2)\otimes\va \in\Ga(\mbs(N))$
where the spinor bundle of $N$ is
\[
\mbs(N)=\big( \mbs(\ell S^1)\op \mbs(\ell S^1)\big)\otimes \mbs(S^1)\cong \C^2.
\]

We point out that $S^1$ has two different spin structures. Recall that a spin
structure is a two-fold covering of the frame bundle $P_{SO}(S^1)$.
Hence we can either take the trivial covering $\sa_1:S^1\times\Z_2\to S^1$
given by two copies of the identity or we may take $\sa_2:S^1\to S^1$ via the mapping
$z\mapsto z^2$ in complex notation. Furthermore, it is interesting to see that
eigenvalues and eigenspinors can be explicitly computed in both structures since
a Fourier decomposition is available in this situation. And particularly, in each spin structure, all the eigenspinors are of constant length. To give an idea of our results, we simply consider the first positive eigenvalue of the Dirac operator $D_{\ig_{S^1}}=i\frac{d}{d\tau}$ on $(S^1,\ig_{S^1},\sa)$, i.e.
\[
D_{\ig_{S^1}}\va_{S^1}=\lm_1\va_{S^1}
\]
with $\lm_1=1$, $\va_{S^1}=e^{-i\tau}$ if $\sa=\sa_1$ or $\lm_1=\frac12$, $\va_{S^1}=e^{-i\tau/2}$
if $\sa=\sa_2$. The proof remains the same if one considers other eigenvalues.

Substituting $\psi=(\psi_1\op\psi_2)\otimes\va_{S^1}$ into the conformal invariant
equation
\begin{\equ}\label{eqs0}
D_{\ig}\psi=|\psi|^2\psi \quad \text{on } N=\ell S^1\times S^1,
\end{\equ}
we are led to an equivalent system of equations
\begin{\equ}\label{eqs1}
\left\{
\aligned
i\frac{d}{dt}\psi_1+i\lm_1\psi_2&=\big( |\psi_1|^2+|\psi_2|^2 \big)\psi_1 \\
-i\frac{d}{dt}\psi_2-i\lm_1\psi_1&=\big( |\psi_1|^2+|\psi_2|^2 \big)\psi_2
\endaligned \right.
\end{\equ}
where $\psi_1,\psi_2:\R/2\pi\ell\Z\to\C$.

Notice that we can write $\psi_1=u_1+iv_1$ and $\psi_2=u_2+iv_2$ for real
functions $u_1,u_2,v_1,v_2$. And moreover, Eq. \eqref{eqs1} is invariant under the multiplication
by $e^{i\vartheta}$ for $\vartheta\in[0,2\pi]$ and the complex conjugation.
Therefore, Eq. \eqref{eqs1} is equivalent to
\begin{\equ}\label{eqs2}
\left\{
\aligned
u'+\lm_1 u&=2(u^2+v^2)v \\
-v'+\lm_1 v&= 2(u^2+v^2)u
\endaligned \right.
\end{\equ}
where $u=u_1=u_2, \, v=v_1=-v_2: \R/2\pi\ell\Z\to(0,\infty)$. Evidently,
Eq. \eqref{eqs2} has an "obvious" constant solution $u=v=\frac{\sqrt{\lm_1}}2$
for all $\ell>0$.

From now on, we are intend to look for non-constant periodic solutions
$u,v$ for Eq. \eqref{eqs2}. Setting $f=2u^2+2v^2$ and $g=2u^2-2v^2$,
we have $uv=\frac{\sqrt{f^2-g^2}}4$ and Eq. \eqref{eqs2} becomes
\begin{\equ}\label{eqs3}
\left\{
\aligned
&g=-\frac1{2\lm_1}f',\\
&2gg'-2ff'=-\frac{2f}{\lm_1}f'\sqrt{f^2-g^2}.
\endaligned \right.
\end{\equ}
After multiplication by $\frac12(f^2-g^2)^{-\frac12}$ in the second equation,
we have
\[
\frac{d}{dt}\big( \sqrt{f^2-g^2}\big)=\frac{d}{dt}\Big(\frac1{2\lm_1}f^2\Big).
\]
Thus, for any solutions $f$ and $g$, there exists a constant $K$ such that
$\sqrt{f^2-g^2}=\frac1{2\lm_1}f^2+K$, that is,
\begin{\equ}\label{eqs4}
g^2=f^2-\Big( \frac1{2\lm_1}f^2+K \Big)^2 \quad \text{and} \quad
\frac1{2\lm_1}f^2+K\geq0.
\end{\equ}

For $K\in\R$, let us denote
\[
F_K(s)=s^2-\Big(\frac1{2\lm_1}s^2+K\Big)^2 \quad \text{for } s\geq0.
\]
Remark that, due to the geometric meaning of Eq. \eqref{eqs0}, the function $f$
defines a conformal metric $\tilde\ig=f^2\ig$ on $\ell S^1\times S^1$.
And this, together with the first equation in \eqref{eqs3}, implies $F_K$
should vanish twice on at some points $s_0,s_1>0$. Thus, the condition on
$K$ is particularly restrictive. In fact, the only possible range is
$K\in(0,\frac{\lm_1}2]$. And, if $K=\frac{\lm_1}2$, we have $f\equiv \lm_1$
and $g\equiv0$ (which correspond exactly the constant solution $u=v=\frac{\sqrt{\lm_1}}2$).

Let $K\in(0,\frac{\lm_1}2)$, and take $0<s_0<s_1$ be the points such that
$F_K$ vanishes. Then the function $F_K$ is positive on the interval $(s_0,s_1)$.
And Eq. \eqref{eqs4} is now equivalent to
\[
\frac{df}{2\lm_1\sqrt{F_K(f)}}=\pm dt,
\]
that is $\eta_K(f)=\pm t + c$, where
\[
\eta_K(f)=\int_{s_0}^{f}\frac{ds}{2\lm_1\sqrt{F_K(s)}}.
\]
Of course, $\eta_K$ is defined on the interval $(s_0,s_1)$. By noting that
$s_0$ and $s_1$ are simple roots of $F_K$, we have $\eta_K$ is well-defined.
Moreover, we have $\eta_K'(s)=\frac1{2\lm_1\sqrt{F_K(s)}}>0$ and
$\eta_K'(s)\to+\infty$ as $s\to s_0$ or $s_1$. Hence $\eta_K$ has an inverse
$\eta_K^{-1}$ which increases from $s_0$ to $s_1$ on the interval
$[0,\eta_K(s_1)]$. And solutions to \eqref{eqs4} can be given by
$f(t)=\eta_K^{-1}(\pm t +c)$ for $c\in\R$.

Setting
\[
f_K(t)=\left\{
\aligned
&\eta_K^{-1}(t) & & t\in[0,\eta_K(s_1)], \\
&\eta_K^{-1}(-t) & & t\in[-\eta_K(s_1),0],
\endaligned \right.
\]
it follows that the positive periodic solutions of Eq. \eqref{eqs3}
can be characterized by $f_K(t+c)$ which is a $2\eta_K(s_1)$-periodic
function. We are looking for solutions having period $2\pi\ell$, that is, functions
whose smallest positive period is of the form $2\pi\ell/k$ for some $k\in\N$. Thus
the non-constant solutions of our problem are functions $f_{K,c}:t\mapsto f_K(t+c)$
for which there exists $k\in\N$ such that
\begin{\equ}\label{period}
\eta_K(s_1)=\frac{\pi\ell}{k}.
\end{\equ}

\begin{Lem}\label{etaK}
$\eta_K(s_1)$ decreases with respect to the factor $K\in(0,\frac{\lm_1}2)$. Particularly,
\[
\lim_{K\to0}\eta_K(s_1)\to+\infty \quad \text{and} \quad
\lim_{K\to\frac{\lm_1}2}\eta_K(s_1)=\frac{\pi}{2\lm_1}.
\]
\end{Lem}
\begin{proof}
To begin with, let us rewrite $F_K$ in its factorization
\[
F_K(s)=\frac1{2\lm_1}(s-s_0)(s_1-s)\Big(s+\frac1{2\lm_1}s^2+K\Big).
\]
in which we have the explicit formulation
\[
s_0=\lm_1-\sqrt{\lm_1^2-2\lm_1K} \quad \text{and} \quad
s_1=\lm_1+\sqrt{\lm_1^2-2\lm_1K}.
\]
Then, we get
\[
\eta_K(s_1)=\int_{s_0}^{s_1}\frac{ds}{\sqrt{(s-s_0)(s_1-s)(s^2+2\lm_1s+2\lm_1 K)}}.
\]
Consider the change of variable $s=s_t=s_0+(s_1-s_0)t$, $t\in[0,1]$, we obtain
\[
\eta_K(s_1)=\int_0^1\frac{dt}{\sqrt{t(1-t)(s_t^2+2\lm_1s_t+2\lm_1 K)}}.
\]

For each $(t,K)\in(0,1)\times(0,\frac{\lm_1}2)$, let's denote
\[
H(t,K)=\frac1{\sqrt{t(1-t)(s_t^2+2\lm_1s_t+2\lm_1 K)}}
\]
Notice $s_t=\lm_1-\sqrt{\lm_1^2-2\lm_1K}(1-2t)$,
it follows from a straightforward calculation that
\[
\frac{\pa}{\pa K}H(t,K)=-H(t,K)^{3}\cdot t(1-t)\cdot\Big[ (s_t+\lm_1)\frac{\pa s_t}{\pa K}+
\lm_1 \Big],
\]
where
\[
(s_t+\lm_1)\frac{\pa s_t}{\pa K}=2\lm_1^2(\lm_1^2-2\lm_1K)^{-\frac12}(1-2t)-\lm_1(1-2t)^2.
\]
Particularly, $\frac{\pa}{\pa K}H(t,K)$ has the following simplified formulation
\begin{\equ}\label{pa H}
\frac{\pa}{\pa K}H(t,K)=\lm_1(\lm_1-2\lm_1K)^{-1} \cdot L(t,K)\cdot H(t,K)-\lm_1 H(t,K)^3\cdot t(1-t)
\end{\equ}
with
\[
L(t,K)=\frac{\big( \lm_1-\sqrt{\lm_1^2-2\lm_1K}(1-2t) \big)^2-\lm_1^2}
{\big( 2\lm_1-\sqrt{\lm_1^2-2\lm_1K}(1-2t) \big)^2-\lm_1^2+2\lm_1K}.
\]
Thus, we see easily that the map $t\mapsto \frac{\pa }{\pa K}H(t,K)$ is in $L^1(0,1)$
for all $K\in(0,\frac{\lm_1}2)$.

Since $t\in(0,1)$, one checks that $L(\frac12,K)\equiv 0$,
\[
L(t,K)<0 \  \text{ if } 0\leq t<\frac12 \quad \text{and} \quad
L(t,K)>0 \ \text{ if } \frac12<t\leq 1.
\]
Moreover, by an elementary computation, we can find
\begin{\equ}\label{LL}
-L(t,K)>L(1-t,K) \quad \text{for all } t\in(0,\frac12).
\end{\equ}
Notice that the values of the function $t\mapsto s_t^2+2\lm_1s_t+2\lm_1 K$ for $t\in[0,\frac12)$
is strictly smaller than that for $t\in(\frac12,1]$. Then, we can substitute \eqref{LL}
into \eqref{pa H} to get
\[
\frac{d}{dK}\eta_K(s_1)<
\lm_1(\lm_1-2\lm_1K)^{-1}\Big(\int_0^{\frac12}L(t,K)\cdot H(t,K)dt
+\int_{\frac12}^1L(t,K)\cdot H(t,K)dt\Big)<0
\]
which shows $\eta_K(s_1)$ is decreasing with respect to $K$.

In order to calculate the limits, let us mention that, as $K\to0$, we have $s=s_0+(s_1-s_0)t\to 2\lm_1 t$ for $t\in[0,1]$. Hence, for arbitrary $\de>0$, it follows from Fatou's lemma that
\[
\aligned
\lim_{K\to0}\eta_K(s_1)&\geq\lim_{K\to0}\int_{\de}^{\frac12}
\frac{dt}{\sqrt{t(1-t)(s^2+2\lm_1s+2\lm_1 K)}} \\
&\geq\frac1{2\lm_1}\int_{\de}^{\frac12}\frac{dt}{t\sqrt{1-t^2}}
>\frac1{2\lm_1}\Big( \ln\frac12-\ln\de \Big).
\endaligned
\]
And thus, by taking $\de\to0$, we have $\lim_{K\to0}\eta_K(s_1)=+\infty$.

For $K\to\frac{\lm_1}2$, we shall use the fact $s_0,s_1\to\lm$ to
obtain
\[
\lim_{K\to\frac{\lm_1}2}\eta_K(s_1)=\frac1{2\lm_1}\int_0^1\frac{dt}{\sqrt{t(1-t)}}
=\frac\pi{2\lm_1},
\]
which completes the whole proof.
\end{proof}

Recall that we are looking for the existence of $2\eta_K(s_1)$-periodic solutions of Eq.
\eqref{eqs3} satisfying \eqref{period}, then Lemma \ref{etaK} implies
\begin{itemize}
\item[$(1)$] For every $\ell>0$, Eq. \eqref{eqs3} has the constant solution
$f_0\equiv \lm_1$ and $g_0\equiv0$ which gives the constant solution
$\psi_1=\frac{\sqrt{\lm_1}}2+i\frac{\sqrt{\lm_1}}2$ and
$\psi_2=\frac{\sqrt{\lm_1}}2-i\frac{\sqrt{\lm_1}}2$ to Eq. \eqref{eqs1}.
Such a solution satisfies
\begin{\equ}\label{f0}
{\rm Vol}(N,f_0^2\ig)=
\int_{N=\ell S^1\times S^1}f_0^2dtd\tau=4\pi^2\lm_1^2\ell.
\end{\equ}
And, for $\ell\leq\frac1{2\lm_1}$, this is the only solution of Eq. \eqref{eqs3}.

\item[$(2)$] Let $d\in\N$ with $\frac{d}{2\lm_1}<\ell\leq \frac{d+1}{2\lm_1}$.
Then for any $k=1,2,\dots,d$, we have $\frac{\pi\ell}k\geq \frac{\pi\ell}d>\frac{\pi}{2\lm_1}$
and there exists $K=K(\ell/k)\in(0,\frac{\lm_1}2)$ such that $\eta_K(s_1)=\frac{\pi\ell}k$.
And the solution $f_k$ of Eq. \eqref{eqs3} corresponding to $K$ satisfies
\begin{\equ}\label{fk}
{\rm Vol}(N,f_k^2\ig)=
\int_{\ell S^1\times S^1}f_k^2 dtd\tau=\frac{2k\pi}{\lm_1}\int_{s_0}^{s_1}
\frac{s^2}{\sqrt{F_K(s)}}ds.
\end{\equ}
\end{itemize}

\begin{Lem}
For any $\ell>\frac1{2\lm_1}$, we have ${\rm Vol}(N, f_1^2\ig)<
\min\big\{{\rm Vol}(N,f_0^2\ig), \ 8\lm_1\pi \big\}$.
\end{Lem}
\begin{proof}
Since $f_1$ has only one period on $\ell S^1$, by Lemma \ref{etaK}, we can fix
$K=K(\ell)>0$ such that $\eta_K(s_1)=\pi\ell$. Following from \eqref{f0}
and \eqref{fk}, we have to show that
\begin{\equ}\label{f1 less}
\frac1{2\lm_1}\int_{s_0}^{s_1}\frac{s^2}{\sqrt{F_K(s)}}ds<\pi\lm_1^2\ell.
\end{\equ}

Similar to the calculations in Lemma \ref{etaK}, let us consider the change of
variable $s=s_t=s_0+(s_1-s_0)t$ for $t\in[0,1]$. Then we have
\begin{eqnarray}\label{eeee}
\pi\lm_1^2\ell-\frac1{2\lm_1}\int_{s_0}^{s_1}\frac{s^2}{\sqrt{F_K(s)}}ds
&=&\int_0^1\frac{\lm_1^2-s_t^2}{\sqrt{t(1-t)(s_t^2+2\lm_1s_t+2\lm_1 K)}}dt
\nonumber \\[0.5em]
&=&\sqrt{\lm_1^2-2\lm_1K}\int_0^1 \frac{(1-2t)\cdot B(t)}{\sqrt{t(1-t)}}dt.
\end{eqnarray}
where
\[
B(t)=\frac{2\lm_1-\sqrt{\lm_1^2-2\lm_1K}(1-2t)}{\sqrt{(s_t^2+2\lm_1s_t+2\lm_1 K)}}.
\]
Clearly, $B(t)>0$ for all $t\in(0,1)$. Furthermore, it follows immediately from
the computations
\[
\frac{(s_t^2+2\lm_1s_t+2\lm_1 K)^{\frac32}}{2\sqrt{\lm_1^2-2\lm_1K}}B'(t)=
(s_t^2+2\lm_1s_t+2\lm_1 K)-\Big( 2\lm_1-\sqrt{\lm_1^2-2\lm_1K}(1-2t) \Big)^2
\]
and
\[
s_t^2+2\lm_1s_t+2\lm_1 K=\Big( 2\lm_1-\sqrt{\lm_1^2-2\lm_1K}(1-2t) \Big)^2-\lm_1^2+2\lm_1K
\]
that $B'(t)<0$. Therefore, we can see from \eqref{eeee} that the total integral
is positive and this implies \eqref{f1 less}. And thus,
\begin{\equ}\label{XXX1}
{\rm Vol}(N, f_1^2\ig)<4\pi^2\lm_1^2\ell={\rm Vol}(N,f_0^2\ig).
\end{\equ}

To give another upper bound for ${\rm Vol}(N, f_1^2\ig)$, let us first set
$\de=\sqrt{\lm_1^2-2\lm_1K}\in(0,\lm_1)$. Then, for the change of
variable $s=s_t=\lm_1-\de(1-2t)$ for $t\in[0,1]$, we have
\[
\frac1{2\lm_1}\int_{s_0}^{s_1}\frac{s^2}{\sqrt{F_K(s)}}ds
=\int_0^1\frac{s_t^2}{\sqrt{t(1-t)(s_t^2+2\lm_1s_t+2\lm_1 K)}}dt
<\int_0^1\frac{Z(t,\de)}{\sqrt{t(1-t)}}dt
\]
where
\[
Z(t,\de)=\frac{s_t^2}{\sqrt{s_t^2+2\lm_1s_t}}.
\]
Notice that, for $x>0$,
\[
\frac{d}{dx}\Big(\frac{x^2}{\sqrt{x^2+2\lm_1x}} \Big)>0, \qquad
\frac{d^2}{dx^2}\Big(\frac{x^2}{\sqrt{x^2+2\lm_1x}} \Big)>0.
\]
And, for $t\in(0,1)$,
\[
\frac{\pa}{\pa\de}Z(t,\de)=\frac{d}{ds_t}\Big(\frac{s_t^2}{\sqrt{s_t^2+2\lm_1s_t}} \Big)
\frac{d s_t}{d\de}=(2t-1)\frac{d}{ds_t}\Big(\frac{s_t^2}{\sqrt{s_t^2+2\lm_1s_t}} \Big).
\]
Hence we have
\[
\frac{d}{d\de}\int_0^1\frac{Z(t,\de)}{\sqrt{t(1-t)}}dt=\int_0^1
\frac{\frac{\pa}{\pa\de}Z(t,\de)}{\sqrt{t(1-t)}}dt>0,
\]
that is, the integral $\int_0^1\frac{Z(t,\de)}{\sqrt{t(1-t)}}dt$ is strictly increasing with
respect to $\de$. At the same time, one calculates easily that
\[
\lim_{\de\to\lm_1}\int_0^1\frac{Z(t,\de)}{\sqrt{t(1-t)}}dt=\int_0^1\frac{2\lm_1 t^2}
{\sqrt{t^2(1-t^2)}}dt=2\lm_1
\]
Therefore, we conclude
\[
{\rm Vol}(N, f_1^2\ig)=\frac{2\pi}{\lm_1}\int_{s_0}^{s_1}
\frac{s^2}{\sqrt{F_K(s)}}ds<4\pi\int_0^1\frac{Z(t,\de)}{\sqrt{t(1-t)}}dt<8\lm_1\pi
\]
and, together with \eqref{XXX1}, we complete the proof.
\end{proof}

\begin{Thm}\label{thm tours}
Let $\lm>0$ denote a positive eigenvalue of the Dirac operator on the second
circle in $N=\ell S^1\times S^1$, then the following facts valid
\begin{itemize}
\item[$(1)$] For every $\ell>0$, the Spinorial Yamabe equation
\begin{\equ}\label{SYT}
D_{\ig}\phi=|\phi|^{2}\phi \quad \text{on } N=\ell S^1\times S^1
\end{\equ}
has a constant length solution
\[
\phi_0=e^{-i(\lm\tau+\vartheta)}\begin{pmatrix}
\frac{\sqrt{\lm}}2+i\frac{\sqrt{\lm}}2\\[0.3em]
\frac{\sqrt{\lm}}2-i\frac{\sqrt{\lm}}2
\end{pmatrix}\in \C^2
\]
for any $\vartheta\in[0,2\pi]$ such that
\[
{\rm Vol}\big( N, |\phi_0|^4 \ig\big)=4\pi^2\lm^2\ell.
\]
And, for $\ell\leq \frac1{2\lm}$, this is the only solution of the form $\phi=\psi e^{-i\lm\tau}\in \mbs(N)$ to Eq. \eqref{SYT}.

\item[$(2)$] Let $\ell>\frac1{2\lm}$ and $d\in\N$ with $\frac{d}{2\lm}<\ell\leq\frac{d+1}{2\lm}$,
Eq. \eqref{SYT} has $d+1$ inequivalent solutions. Particularly, these solutions are given by the constant length solution and $k$ periods of a solution $\phi_{\ell,k}$ on $N$ with
fundamental period $\frac{2\pi\ell}{k}$ for $k=1,2,\dots,d$.

\item[$(3)$] Let $\ell>\frac1{2\lm}$ and $\phi_{\ell,1}$ denote the $2\pi\ell$-periodic
solution of Eq. \eqref{SYT}, then
\[
{\rm Vol}\big(N, |\phi_{\ell,1}|^4\ig \big)<\min\big\{ {\rm Vol}\big( N, |\phi_0|^4 \ig\big),\
8\pi\lm \big\}.
\]
and
\[
\lim_{\ell\to\infty}{\rm Vol}\big(N, |\phi_{\ell,1}|^4\ig \big)=8\pi\lm.
\]

\item[$(4)$] Let $\sa_N^*$ denote the nontrivial spin structure on $N=\ell S^1\times S^1$ such that
$\lm=\frac12$ is the first positive eigenvalue of the Dirac operator on the second circle, then
\[
\lm_{\min}^+(N,\ig,\sa_N^*)\leq {\rm Vol}\big(N, |\phi_{\ell,1}|^4\ig \big)^{\frac12}<2\sqrt\pi
\]
for all $\ell>0$.
\end{itemize}
\end{Thm}

\vspace{2mm}
{\sc Yannick Sire\\
Department of Mathematics, Johns Hopkins University,\\
3400 N. Charles Street, Baltimore,
Maryland 21218}\\
sire@math.jhu.edu\\

{\sc Tian Xu\\
 Center for Applied Mathematics, Tianjin University,\\
 300072, Tianjin, China}\\
 xutian@amss.ac.cn

\end{document}